\documentclass[reqno]{amsart}
\UseRawInputEncoding
\usepackage{amsmath,,amsfonts,amssymb,amsthm,amscd,latexsym,cite}
\usepackage{mathrsfs}
\usepackage{color,comment}
\newtheorem{thm}{Theorem}[section]
\newtheorem{theorem}[thm]{Theorem}

\newtheorem{lemma}[thm]{Lemma}
\newtheorem{cor}[thm]{Corollary}

\newtheorem{proposition}[thm]{Proposition}
\newtheorem{definition}[thm]{Definition}
\newtheorem{remark}[thm]{Remark}
\newtheorem{rem}[thm]{Remark}

\newcommand{\cA}{{\mathcal A}}

\newcommand{\cE}{{\mathcal E}}

\newcommand{\cH}{{\mathcal H}}

\newcommand{\cM}{{\mathcal M}}

\newcommand{\cP}{{\mathcal P}}

\newcommand{\cZ}{{\mathcal Z}}

 \usepackage{color}
 \newcommand{\norm}[1]{\left\lVert#1\right\rVert}
\usepackage{hyperref}					
\hypersetup{colorlinks,
	linkcolor=blue,%
	citecolor=blue}

 \usepackage{float}
\begin{document}
\title[Norms of generalized derivations]{Norms of generalized derivations with values in symmetric spaces}


\author[J. Huang]{J. Huang}
\address[Jinghao Huang]{Institute for Advanced Study in Mathematics, Harbin Institute of Technology, Harbin 150001, China \emph{E-mail~:} {\tt jinghao.huang@hit.edu.cn} }

\author{M. Pliev}
\address[Marat Pliev]{Southern Mathematical Institute of the Russian Academy of Sciences, Vladikavkaz, 362025 Russia,
North-Caucasus Center for Mathematical
Research of the Vladikavkaz Scientific Centre of the Russian Academy
of Sciences, Vladikavkaz, 362027 Russia  and North-Ossetian State  University, Vladikavkaz, 362025 Russia}
\email{plimarat@yandex.ru}

\author[F. Sukochev]{F. Sukochev}
\address[Fedor Sukochev]{School of Mathematics and Statistics, University of New South Wales, Kensington, 2052, NSW, Australia  \emph{E-mail~:} {\tt f.sukochev@unsw.edu.au} }

\author[R. Xu]{R. Xu}
\address[Ran Xu]{Institute for Advanced Study in Mathematics, Harbin Institute of Technology, Harbin 150001, China \emph{E-mail~:} {\tt xxurann@stu.hit.edu.cn} }

\thanks{\footnotesize
J. Huang was supported by the NNSF of China (No. 12031004, 12301160 and  12471134).
F. Sukochev was supported by  the Australian Research Council (DP230100434). 
M. Pliev  was supported by the Ministry
of Science and Education of the Russian Federation (number of the agreement is 075--02--2026--738)
}

\subjclass[2010]{46L57; 46L10; 46L52; 46E30. \hfill 
}

\keywords{generalized derivation; von Neumann algebra;  non-commutative symmetric space.
}

\begin{abstract} Let $B(\cH)$ be the $*$-algebra of all bounded linear operators on a Hilbert space $\cH$.  
A classical result due to Stamplfi shows that, for any $a\in B(\cH)$,  the norm of the inner derivation $\delta_a=[a,\cdot]$ on $B(\cH)$ is given by $\norm{\delta_a}_{B(\cH)\to B(\cH)}=2\inf\{  \norm{a-c{\bf 1}}_{B(\cH)}:c\in \mathbb{C}\}$.
In the present paper, we provide a formula for the norm of a generalized derivation implemented by self-adjoint operators from an ideal in $B(\cH)$, which partially answers a question posed by Fialkow and Loebl in the 1980s.
We also consider this question in a    more general setting  of  a  symmetrically normed (not necessarily complete or separable) space  affiliated with a properly infinite semifinite von Neumann algebra.  
\end{abstract}
\maketitle

\section{Introduction}

The study of derivations on $C^*$-/$W^*$-algebras has been
at the heart of the study of operator algebras  since the inception
of the field, see \cite{HS24} and references therein. 
Let \( A \) be an associative algebra over the field \( \mathbb{C} \) of complex numbers and let \( M \) be an \( A \)-bimodule. 
A linear map \( \delta \) from \( A \) into \( M \) is called a derivation if
\[
\delta(xy) = \delta(x)y + x\delta(y) \quad \text{for all } x, y \in A.
\]
It is clear that each element \( a \in M \) defines a derivation \( \delta_a \) from \( A \) into \( M \) by
\[
\delta_a(x) = [a, x] = ax - xa, \quad x \in A.
\]
Such derivations \( \delta_a \) are called  inner derivations.
It is well-known that
all derivations on 
a $W^*$-algebra (in particular, 
$B(\cH)$)
  are inner \cite[Theorem 4.1.6]{Sakai}.
Stampfli \cite{Stam} obtained the   following
remarkable formula for the norm of an inner derivation $\delta_a:=[a, \cdot ]$ on $B(\cH)$:
\begin{align*}
\left\| \delta_a \right\|_{B(\cH)\to B(\cH)} =2\inf \left\{
 \left\| a- c{\bf 1}
\right\|_{B(\cH)}:c \in \mathbb{C}   
\right\},\,\,  a\in B(\cH).
\end{align*}
Later on, a lot of effort has been spent  to extend this
result to inner derivations  on general von Neumann algebras and
other operator spaces (see \cite{Gajendragadkar,Zsido,Mc,Ch,KLR} and references therein). 
It is natural to ask \emph{whether there are  Stampfli type formulas for the norm of an inner derivation implemented by an operator from a noncommutative symmetric space}. 
An analogue of Stampfli's formula for a skew-adjoint
derivation $\delta_a\colon \cM\to \cM_*$ from a  von
Neumann algebra $\cM$ into its predual $\cM_*$ was established in  \cite[Theorems~1.1 and 3.1]{BHS}, i.e., 
\begin{align}\label{disS2}
\norm{\delta_a}_{\cM\to \cM_*} =2 \inf \{\norm{a-z}_{\cM_*}:z\in Z(LS(\cM))\},
\end{align}
where $Z(LS(\cM))$ stands for the center of the space of all locally measurable operators affiliated with $\cM$. 
In contrast to Stampfli's result, 
the predual version only holds for self-adjoint operators, see \cite{BBS} for counterexamples in the setting of normal (non-self-adjoint) operators.

Let $a,b$ be bounded operators on a Hilbert space $\cH$. 
The generalized derivation $\delta_{a,b}$ is defined by $$\delta_{a,b}(x)=ax-xb, \quad x\in B(\cH).$$
This notion was introduced by Rosenblum\cite{Rosenblum}. 
Assume that  
 $J$ is a proper two-sided ideal of $B(\cH)$,   $a,b\in J$ and $\delta_{ a,b}$ is an  operator from $B(\cH)$ into $J$. 
In \cite[p.577]{FL}, Fialkow and Loebl asked  
\begin{quote}
\emph{what is the norm $\norm{\delta_{ a,b}}_{B(\cH)\to J} $ of $\delta_{ a,b}$?}
\end{quote}
Here, $J$ is a norm (proper) ideal in $B(\cH)$ in the sense of Schatten\cite[p.57]{Schatten}.
In general, it is difficult to calculate this norm while 
 an estimate for the case when $a=b$ is given in \cite{FL, F79}:
$$
2\cdot  \inf_{\lambda\in \mathbb{C} } \norm{a-\lambda}_{J} \ge \norm{\delta_a }_{B(\cH)\to J} \ge {\rm diam}(W(a)),$$
where $W(a)$ denotes the numerical range of $a$. 
For general  operators $a,b$ from   norm ideals $J$ in $B(\cH)$, there exists no precise formula for $\norm{\delta_{a,b}}_{B(\cH)\to J}$ in literature, and  formula
\eqref{disS2} fails  for  general norm  ideals $J\subset K(\cH)$ (here, $K(\cH)$ stands for the set of all compact operators on $\cH$). 
The main purpose of the present paper is to answer the question by Fialkow and Loebl in the setting of self-adjoint operators in a proper operator ideal in $B(\cH)$.

We briefly recall the connection between ideals of compact operators and symmetric sequence spaces, which is known as the Calkin correspondence. 
Recall that 
an ideal $ J$
 of the algebra $B(\cH)$  is said to be symmetrically normed if
$\{\mu_k(x)\}_{k=1}^\infty\leq\{\mu_k(y)\}_{k=1}^\infty$
and $y\in J$
 implies that $\left\|x\right\|_ J  \leq \left\|y\right\|_{J}$
\cite{DPS,GK}, 
where $\mu(x):=\{ \mu_1(x),\mu_2(x),\cdots \}$ stands for the decreasing sequence of generalized singular values of $x$.
The Calkin correspondence shows that there exists a bijection between (symmetrically normed) two-sided ideals $C_E$ of compact operators and  symmetric sequence spaces $E$ in $c_0$ (see e.g. \cite[Theorem 1.2.3]{LSZ}):
$$x\in C_E \Leftrightarrow   \mu(x)\in E. $$ 
For the original Calkin correspondence (under some restrictions, e.g., separability, Fatou norm\footnote{Recall that the norm on a symmetrically normed operator ideal $J$ 
is called a Fatou
norm if the unit ball of $J$ is closed with respect to strong (or, equivalently, weak) operator
convergence.} imposed on  the symmetrically normed ideal $C_E$), we refer to \cite{GK,Simon,Schatten}. 
 Note that, Fialkow and Loebl\cite{FL} considered 
general (complete)
 symmetrically normed ideals of compact operators which didn't possess any additional properties like Fatou norm or seprability of the ideal.
In order to treat Fialkow and Loebl's question in the setting of general two-sided Banach ideals in $B(\cH)$, 
we need the  bijective correspondence
between general symmetric operator spaces/operator ideals and general symmetric function/sequence spaces which was established in  \cite{Kalton_S}
through 
the notion of the so-called uniform submajorization (see also detailed discussion of the latter notion in \cite{LSZ}).


\begin{theorem}\label{special}Let $E \subset c_0$   be a symmetrically normed sequence  space. For any self-adjoint operators $a,b \in C_E$ (the operator ideal   corresponding to  $E$), we have 
\begin{align*}
\norm{\delta_{a,b}}_{B(\cH)\to C_E}=
 \norm{  \left( \mu(a_+ )+\mu(  b_- )\right) \oplus \left( \mu(b_+ )+\mu(  a_-   )\right) }_{E   }.
\end{align*}
 Here, $\mu(x)\oplus \mu(y)$ is a sequence in $\ell_\infty \oplus \ell_\infty$ which can be identified as an element in $\ell_\infty$.
\end{theorem}

We obtain Theorem \ref{special} as an immediate corollary of the following more general result, Theorem \ref{main2} below. 
Let 
$a$ and $b$ be  self-adjoint 
 $\tau$-compact operators affiliated with a semifinite factor $\cM$
(if  $\cM=B(\cH)$, then $a$ and $b$ are self-adjoint compact  operators from some proper ideal in  $B(\cH)$). 
 It is worth noting that  derivations
from a semifinite von Neumann algebra $\cM$ into an arbitrary
symmetric space $E(\cM,\tau)$ of $\tau$-measurable operators affiliated with $\cM$ are
necessarily inner  \cite{BCS}.
Consequently, see Proposition ~\ref{gen_inner} below,
 for any generalized derivation $\delta_{a,b}:\cM\to E(\cM,\tau)$, 
there exists $z$ in the center $ Z(S(\cM,\tau))$ of the algebra $S(\cM,\tau)$ of all $\tau$-measurable operators affiliated with $\cM$ such that $a-z,b-z\in E(\cM,\tau)$ and $\delta_{a,b}=\delta_{a-z,b-z}$.
 
\begin{theorem}\label{main2}Let $E(0,\infty)\subset S_0(0,\infty )$ (functions whose  decreasing rearrangements vanish  at infinity) be a symmetrically normed space and let $\cM$ be an infinite factor equipped with a semifinite faithful normal trace $\tau$. For any self-adjoint operators $a,b \in E(\cM,\tau)$ (the noncommutative  space corresponding to  $E(0,\infty)$), we have 
\begin{align*}
\norm{\delta_{a,b}}_{\cM\to E(\cM,\tau)}=
 \norm{  \left( \mu(a_+ )+\mu(  b_- )\right) \oplus \left( \mu(b_+ )+\mu(  a_-   )\right) }_{E (0,\infty)   }.
\end{align*}
Here, $\mu(x)$ stands for the generalized singular value function of a $\tau$-measurable operator $x$ affiliated with $\cM$.
\end{theorem}

Let $E$ be a Banach symmetric sequence space~\cite{KPS,LT1} and let $C_E$ be the symmetric ideal in $B(\cH)$ 
 generated by $E$ (via the Calkin correspondence), see \cite{LSZ,Kalton_S}.
In \cite{Kittaneh}, Kittaneh proved that if $a,b\in B(\cH)$ are
positive, and $x\in B(\cH)$, then 
$$\mu_j(ax-xb) \le \norm{x}_{B(\cH)}\mu_j(a\oplus b),$$
and therefore,
\begin{align}\label{kit}
\norm{ax-xb}_{C_E} \le \norm{x}_{B(\cH)} \norm{a\oplus b}_{C_E},
\end{align}
where 
    the direct sum notation   $a\oplus b$  denotes  the block-diagonal operator
   $ \left(
      \begin{array}{cc}
        a & 0 \\
        0 & b \\
      \end{array}
    \right)$
defined on $\cH \oplus  \cH$.

Specializing $\cM=B(\cH)$, 
Theorem~\ref{special} is an immediate consequence of  Theorem~\ref{main2}.
The proof of the above theorem is  completely different from those in  \cite{Stam,BHS,Zsido}.
One of the  main ingredients in the proof  is
an extension of    inequality  \eqref{kit} (see Theorem \ref{upper estimate} below)
 to the setting of self-adjoint operators, which relies on the notion of uniform submajorization. 
 Another important auxiliary result is an analogue of 
the commutator estimates in    \cite{BS2012,BS12b} (see Lemma \ref{lemma:approx}).

Our second main theorem treats positive operators affiliated with  properly infinite semifinite von Neumann algebras, which is not necessarily a factor. 
%

\begin{thm}\label{thm2}
Let $E(0,\infty)\subset S_0(0,\infty )$ be a symmetrically normed function space.
Let
 $\cM$ be a properly infinite von Neumann algebra equipped with a semifinite faithful normal trace and let $a ,b \in E(\cM,\tau)$ be positive operators.
Then,
 \begin{align}\label{===1}
\norm{\delta_{a,b } }_{\cM\to  E(\cM,\tau)}  =\norm{a\oplus b }_{E(\cM\oplus \cM,\tau \oplus \tau)}.
\end{align}
\end{thm}
In the special cases  when $\cM$ is a factor or is properly infinite, Theorems \ref{main2} and~\ref{thm2} deliver an alternative  proof  of  the main results in \cite{BHS}. 
When $\cM=B(\cH)$, equality \eqref{===1} sharpens   inequality \eqref{kit}.

\section{Preliminaries}

In this section,
we state  some basic facts and notions that we need
for the further presentation. 
General information on von Neumann
algebras and noncommutative symmetric  spaces  can be found in
\cite{LSZ,KPS,LT2}.
\subsection{$\tau$-measurable operators}
Let $I = (0,a)$ with  $a\in (0,\infty]$,   equipped with  the
Lebesgue measure $m$. Denote by $\Sigma$  the $\sigma$-algebra of Lebesgue measurable subsets of $I$. By $(I,m)$ we denote the measure space $(I,\Sigma,m)$. 
Let $S(I,m)$ (or $S(I)$ for brevity) be the space of all equivalence classes of almost everywhere finite Lebesgue measurable
real-valued (or complex-valued) 
functions on $I$. For $x\in S(I)$, we denote by
$\mu(x)$ the decreasing rearrangement of the function $|x|$\cite{KPS,LSZ,LT2}. That is,
$$
\mu(t; x)=\inf \left\{s\geq0:\ m(\{|x|>s\})\leq t \right\},\quad t>0.
$$
We denote by $S_0(0,\infty )$ the set $\{x\in S(0,\infty )\mid \lim\limits_{t\to \infty}\mu(t;x)=0\}$.

Let $\mathcal{M}$ be a   von Neumann algebra on a   Hilbert space $\cH$ with  the identity ${\bf 1}$.
Let $P(\mathcal{M})$ denote the lattice of all projections in $\mathcal{M}$ and
$U(\cM)$ denote the  set of all unitary elements in $\cM$.


We say that an operator   $x$ is \emph{measurable} (denoted by $x  \in S(\cM)$) if  $x$ is closed,
densely defined, affiliated with $\cM$ and $e 
^{|x|}
(\lambda,\infty )$ is a finite projection for some $\lambda >0$, where $e ^{|x|}$ stands for the spectral measure of $|x|$. It follows
immediately that in the case when $\cM$ is a von Neumann algebra of type $III$ or a type $I$ factor,
we have $S(\cM) =\cM$. For type $II$ von Neumann algebras, this is no longer true\cite{LSZ,DPS}.
We use the notation  \( l(x) \), \( r(x) \) and  \( c(x) \) to denote the   left support, the right support, and  the central support respectively of a measurable operator  \( x  \).
If, in addition,  $x$ is self-adjoint, then   we have  \( l(x) =  r(x) \). In this case, we denote both  the  left and the  right supports  by  \( s(x) \).

Let $\cM$ be a semifinite von Neumann algebra on a   Hilbert space $\cH$ equipped with a faithful normal semifinite trace $\tau$.
A closed and densely defined operator $A$ affiliated with $\mathcal{M}$ is called \emph{$\tau$-measurable} if
$\tau(e ^{|x|}(s,\infty))<\infty$ for sufficiently large $s$.
 We denote the set of all $\tau$-measurable operators by
$S(\mathcal{M},\tau)$\cite{LSZ,DPS}.
For every $x\in S(\mathcal{M},\tau),$ we define its singular value function $\mu(x)$ by setting
$$\mu(t; x)=\inf\{\norm{x({\bf 1}-p)}_{ \infty   }:\forall p\in P(\mathcal{M}),\quad \tau(p)\leq t\}, \quad t>0.$$
We denote by $S_0(\cM,\tau)$ the set $\{x\in S(\cM,\tau)\mid \lim\limits_{t\to \infty}\mu(t;x)=0\}$.

A closed operator $x$ affiliated with $\cM$ is called \emph{locally measurable} if there exists a sequence $\{z_n\}^\infty
_{n=1}$ of
central projections in $\cM$ such that $z_n \uparrow {\bf 1}$ and $xz_n \in S(\cM)$ for any $n \in \mathbb{N}$. The collection of
all locally measurable operators with respect to $\cM$ is denoted by $LS(\cM)$, which is a unital
$*$-algebra with respect to strong sums and products (denoted simply by $x + y$ and $xy$ for all
$x,y \in  LS(\cM))$\cite{BS12b,BCLSZ,BCS}.


\subsection{Symmetric spaces}

\begin{definition}
We say that $(E(I),\left\|\cdot\right\|_E) $ (or    $(E,\left\|\cdot\right\|_E)$ for brevity) is a symmetrically normed   function space
on $I$ if the following hold:
\begin{enumerate}
\item $E$ is a subspace of $S(I);$
\item $(E,\left\|\cdot\right\|_E)$ is a normed space;
\item If $x\in E$ and if $y\in S(I)$  are such that $\mu(y)\leq\mu(x),$ then $y\in E$ and $\left\|y\right\|_E\leq \left\|x\right\|_E.$
\end{enumerate}
If $E$ is a Banach space, then it is called a symmetric function space. 
If  for any $x,y\in E$ and $\int_{0}^t \mu(s;y)ds\le \int_{0}^t \mu(s;x)ds$, $t>0$ (denoted by $y\prec \prec x$), implies that $\norm{y}_E\le\norm{x}_E$, then $E$ is said to be a strongly  symmetrically normed  space (if, in addition,  $E$ is a Banach space, then it is called a strongly symmetric space).
\end{definition}
For the general theory of symmetric function  spaces, we refer the reader to \cite{LT2,DPS,KPS}.

\begin{definition}
Let $  \cE$   be a linear   subspace of    $S(\mathcal{M},\tau)$ equipped with a 
norm $\left\|\cdot\right\|_{\cE}$.
We say that
$ \cE $ is a symmetrically normed
operator space if for $ x\in \cE $  and for every $y\in S(\mathcal{M},\tau)$ with
$\mu(y)\leq\mu(x),$ we have $y\in \cE $
and $\left\|y\right\|_{\cE} \leq
\left\|x\right\|_{ \cE} $.
If $\cE$ is a Banach space, then it is called a symmetric operator space. 
If  for any $x,y\in \cE$ and   $y\prec \prec x$, implies that $\norm{y}_\cE\le\norm{x}_\cE$, then $\cE$ is said to be a strongly symmetrically normed space (if, in addition, $\cE$ is Banach, then it is called a strongly symmetric space).
\end{definition}

Recall the following  construction of a symmetric   operator space (or
non-commutative symmetric  Banach space) $E(\mathcal{M},\tau)$.
Let
$E$ be a symmetrically normed   function   space on $(0,\tau({\bf 1}))$.
Set
$$
E(\mathcal{M},\tau)=\Big\{x\in S(\mathcal{M},\tau):\ \mu(x)\in E\Big\}
.$$
The  natural norm on $E(\mathcal{M},\tau)$ is defined by
$$
\left\|x\right\|_{E (\mathcal{M},\tau)}:= \left\|\mu(x)\right\|_E,\quad x \in E(\mathcal{M},\tau).
$$
We note that $E(\mathcal{M},\tau)$  is a  Banach space with respect to $\left\|\cdot\right\|_{E (\mathcal{M},\tau)}$ if $E$ is Banach, and it is called the (non-commutative) symmetric operator space associated
with $(\mathcal{M},\tau)$ corresponding to $(E,\left\|\cdot\right\|_{E})$ \cite{Kalton_S,LSZ,DPS}.

Recall that derivations from a semifinite von Neumann algebra $\cM$ into a symmetric space $E(\cM,\tau)$ are inner\cite{BCS}. 
 For generalized derivations, we have a similar result. 
 
 \begin{proposition}\label{gen_inner} Let $\cM$ be a semifinite von Neumann algebra equipped with a faithful normal trace $\tau$. 
 Assume that $E(\cM,\tau)$ is one of the followings:\begin{enumerate}
                                                             \item   an (not necessarily normed) ideal  in $\cM$;
                                                             \item   a Banach $\cM$-bimodule (see \cite[Chapter 4]{DPS} or \cite[Section 6]{BCS} for the definition) of locally measurable operators affiliated with $\cM$; 
                                                             \item     an $\cM$-bimodule of locally measurable operators affiliated with a properly infinite von Neumann algebra $\cM$.
                                                           \end{enumerate}
 Let $a,b\in S(\cM,\tau)$ be such that $\delta_{a,b}$ is bounded from $\cM$ into $E(\cM,\tau)$.
 There exists $z\in Z(S(\cM,\tau))$ such that $a-z,b-z\in E(\cM,\tau)$ 
 and $\delta_{a,b}=\delta_{a-z,b-z}$. 
 \end{proposition}
\begin{proof}
Observing that  $a-b= \delta_{a,b}({\bf 1})\in E(\cM,\tau)$, we have 
$$\delta_a(x) = \delta_{a,b}(x) -x(a-b) \in E(\cM,\tau), \,\, x\in\cM.$$
  By \cite[Corollary 3]{BS12b} (respectively, \cite[Theorem 6.10]{BCS}, \cite[Theorem 6.3]{BCS}  or \cite[Theorem 5.3.7, Corollary 7.6.3, Theorem 7.7.7]{BCLSZ}), 
  $\delta_a $ is an inner derivation, i.e., there exists $a'\in E(\cM,\tau)$ such that $\delta_a=\delta_{a'}$.
  Hence,  
   $a-a' $ commutes with all elements in $\cM$ and therefore, commutes with all elements in $S(\cM,\tau)$ (see e.g. \cite[Proposition 2.2.22]{DPS}).
   Since $a,a'\in S(\cM,\tau)$, it follows that  
   $z:= a-a'\in Z(S(\cM,\tau))$.     
  Since  $a-z=a' \in E(\cM,\tau)$ and  $$(a-z)x -x (b-z) =ax-xb \in E(\cM,\tau)$$ for any $x\in \cM$, it follows that $x(b-z)\in E(\cM,\tau)$ for any $x\in \cM$. 
  In particular, $b-z\in E(\cM,\tau)$.
\end{proof}

 \section{Norm estimate of a generalized derivation}

Fialkow and Loebl's question concerns   ideals with symmetric norms, which are not necessarily separable. 
The main tool which allows to extend Calkin correspondence to  general  symmetrically normed
spaces is the so-called \emph{uniform submajorization} introduced in  \cite{Kalton_S} (see also \cite{LSZ}):
for $x,y\in S(\cM,\tau)$, $y$ is said to be uniformly submajorized by $x$ (written $y \vartriangleleft x$) if there exists $\lambda \in \mathbb{N}$ such that
$$\int_{\lambda t_1}^{t_2} \mu(s;y)ds \le \int_{t_1}^{t_2} \mu(s;x)ds ,~\forall ~ 0\le \lambda t_1 \le t_2 . $$
 If $t >0,$ the dilation operator $\sigma_{t}$ is defined by setting 
  $$(\sigma_{t}x)(s)=x\left(\frac{s}{t}\right),~s>0,$$ for any $x\in S(0,\infty )$. 
Observe that, for any $0<b \le 1$ and $x\in L_1(\cM,\tau)+\cM$, 
we have 
$$\int_{\frac{1}{b} t_1 }^{t_2} b\mu(s;x)ds \le  \int_{\frac{1}{b} t_1}^{\frac{1}{b} t_2 }  \mu(s ;x) d(b  s) = \int_{t_1}^{t_2}   \mu(s/b;x) ds, $$
i.e., 
\begin{align}\label{dilationuniform}
 b\mu(x) \vartriangleleft \sigma_b \mu(x) .  \end{align}

\begin{proposition}\label{direct sum}Let $\cM$ be a semifinite von Neumann algebra equipped with a semifinite faithful normal trace $\tau$. 
Let $x_1,y_1,x_2,y_2\in L_1(\cM,\tau)+\cM$
be such that $x_1 \vartriangleleft x_2$ and $y_1\vartriangleleft y_2$.
For any $\varepsilon>0$, we have 
$$(1-\varepsilon ) (x_1\oplus y_1) \vartriangleleft x_2 \oplus y_2.$$
\end{proposition}
\begin{proof}
By \cite[Theorem 3.4.2]{LSZ},
for any $\varepsilon>0$, there exists an integer   $n\ge 1$ such that   
$$(1-\varepsilon )\mu(x_1) = \sum_{k=1}^n \frac1n  a_k,    $$
and
$$(1-\varepsilon )\mu(y_1) = \sum_{k=1}^n  \frac1n  b_k ,   $$
where  $0\le a_k,b_k\in (L_1+L_\infty)(0,\infty)$,~ $\mu(a_k )\le \mu(x_2)$ and $\mu(b_k)\le \mu(y_2)$, ~$k=1,\cdots ,n$.
Therefore,
$$(1-\varepsilon)(\mu(x_1)\oplus \mu(y_1)) 
= \frac{1}{n } \sum_{k=1}^n  \left(  a_k\oplus  b_k  \right),$$
where  $\mu(a_k \oplus b_k)
	\le 
\mu(x_2\oplus y_2) $  \cite[Remark 3.2]{HanS}.
By \cite[Theorem 3.4.2]{LSZ}, we have 
$$(1-\varepsilon)(\mu(x_1)\oplus \mu(y_1))  \vartriangleleft  x_2\oplus y_2,$$
which completes the proof.
\end{proof}

As mentioned in the Introduction, the following theorem plays an important role in the proof of Theorem \ref{main2}.
It provides an upper estimate for the norm of a generalized derivation. 

\begin{theorem}\label{upper estimate}
Let $\cM$ be a semifinite von Neumann algebra equipped with a semifinite faithful normal trace $\tau$. 
Let $E(0,\infty )$ be a symmetrically normed function space. 
For 
any self-adjoint $ a,b\in E(\cM,\tau)$, we have 
\begin{align*} \norm{\delta_{a,b} }_{\cM\to E(\cM,\tau)}\le
 \norm{(\mu(a_+) +\mu(  b_- )) \oplus (\mu(a_- )+\mu( b_+ )) }_{E( 0,\infty )}.
\end{align*}
\end{theorem}
\begin{proof}
By the Russo--Dye theorem\cite[Corollary 1]{RD}, it suffices to prove 
that 
\begin{align*} \norm{\delta_{a,b}(u) }_{  E(\cM,\tau)}\le
 \norm{(\mu(a_+) +\mu(  b_- )) \oplus (\mu(a_- )+\mu( b_+ )) }_{E( 0,\infty )}
\end{align*}
for all unitary elements $u\in \cM$.

For any unitary operator   $u\in \cM$, 
we have
\begin{align*}
\mu(\delta_{a,b}(u))= \mu(a-ub u^*) &\,\,\qquad =\qquad   \mu(a_+-a_-  -ub_+u^*+ub_-u^*)\\
&\stackrel{\mbox{\tiny \cite[Thm.3.4(1)]{HanS}}}{\le}
\mu((a_+ + ub_-u^* )\oplus (a_-+ub_+ u^*)) \\
&\,\,\qquad  =\qquad \mu(  \mu(a_+ + ub_-u^*)\oplus \mu(a_-+ub_+ u^*)) .
\end{align*}
By  \cite[Lemma  3.4.4]{LSZ},
we have $ \mu(a_+ + ub_-u^*)\vartriangleleft \mu(a_+) +\mu( ub_-u^* )$ and $\mu(a_-+ub_+ u^*) \vartriangleleft \mu(a_- )+\mu( ub_+ u^*)$. 
Therefore, by Proposition \ref{direct sum}, for any $\varepsilon >0 $,  we have 
\begin{align*} (1-\varepsilon ) \mu(\delta_{a,b}(u))& \vartriangleleft   (\mu(a_+) +\mu( ub_-u^*)) \oplus (\mu(a_- )+\mu( ub_+ u^*))\\
&  =   (\mu(a_+) +\mu(  b_- )) \oplus (\mu(a_- )+\mu(  b_+ )) .
\end{align*}

Since $\varepsilon$ is arbitrarily taken, it follows from \cite[Corollary 3.4.3]{LSZ} that 
$$\norm{\delta_{a,b}(u)}_{E(\cM,\tau)}\le  \norm{(\mu(a_+) +\mu(  b_- )) \oplus (\mu(a_- )+\mu( b_+ )) }_{E( 0,\infty   )}.$$
Therefore, we have
\begin{align*} \norm{\delta_{a,b} }_{\cM\to E(\cM,\tau)}&=
\sup_{u\in U(\cM) } \norm{\delta_{a,b}(u)}_{E(\cM,\tau)}\\
&\le
 \norm{(\mu(a_+) +\mu(  b_- )) \oplus (\mu(a_- )+\mu( b_+ )) }_{E( 0,\infty )}.
\end{align*}
The proof is complete.
\end{proof}
\section{Proof of Theorem \ref{main2}}

In this   section, 
we consider Fialkow's question (see Introduction) in the setting of self-adjoint operators. 

The following lemma  is certainly  known to experts. Due to the lack of a suitable reference, we include a complete proof below. 
\begin{lemma}\label{lemma:decomposition}
Let $\cM$ be an atomless von Neumann algebra equipped with a semifinite faithful normal trace $\tau$.
Let $x\in S(\cM,\tau)$ such that $\mu(x)$ is continuous on $(0,\tau(s|x|))$  and let  $N \ge 2$ be a natural number. 
There exists $0 <\cdots<s_i<s_{i+1} < \cdots < \tau(s(|x|))$, $-\infty <i<\infty $, 
such that 
$$ \mu(s_i;x )  \ge \mu(s_{i+1};x )> \frac{N-1}{N}  \mu(s_i;x) . $$
\end{lemma}
\begin{proof}

For any $0< a<b<\tau(s(|x|))$, we have  
 $$ \mu(a;x)<\infty \mbox{
 and }\mu(b;x) >0.$$ 
 Consider the following cover
\begin{equation*}
 \begin{split}  \Big[
 a, b \Big]  \subset \bigcup_{s\in (0, \tau(s(|x|))] }  \left\{   t \in (s,\tau(s(|x|)) )  : 
        \mu(s;x)  \ge \mu(t;x)> \frac{N-1}{N}  \mu(s; x) 
     \right\}
     \end{split}
\end{equation*}
By the right-continuity of singular value functions\cite[Chapter 3.2]{DPS},   it is an open cover. 
Since $[a, b] $ is compact, it follows that it has a finite cover, i.e., there exist 
finitely many 
$s_i$'s in $(0, \tau(s(|x|))]$ such that  
\begin{equation*}
 \begin{split} 
 \left[
a, b \right] 
 \subset ~ \bigcup_{i }  \left\{   t \in (s_i,\tau(s(|x|)) )  : ~\begin{split}
        \mu(s_i;x)   \ge \mu(t;x)> \frac{N-1}{N}  \mu(s_i;x)  
     \end{split} 
     \right\}.
     \end{split}
\end{equation*}
Taking $a\to 0$ and $b \to \tau(s(|x|))$,
we obtain  $0 <\cdots<s_i<s_{i+1} < \cdots < \tau(s(|x|))$, $-\infty <i<\infty $, 
such that 
$$ \mu(s_i;x )  \ge \mu(s_{i+1};x )> \frac{N-1}{N}  \mu(s_i;x) . $$
 This completes the proof.
\end{proof}

Recall from \cite[Remark 6.5.4]{KR-II}  that 
atomic type $I_{\infty}$ factors   are $*$-isomorphic to $B(\cH)$, where $\cH$ is an infinite dimension Hilbert space. 
 Fialkow's question is naturally placed in the setting of infinite factors. 
Throughout the proof of Theorem ~\ref{main2}, we focus on the atomless infinite factor case. 
The atomic case follows by a similar argument.

The following lemma is the key ingredient of the proof for Theorem \ref{main2}. 
The main tool is the total comparability of projections in a factor\cite[Proposition 6.2.6]{KR-II}.
\begin{lemma}\label{lemma:approx}Let $\cM$ be a    factor equipped with a semifinite infinite faithful normal trace $\tau$ and    $a, b\in S_0(\cM,\tau)$ be self-adjoint elements such that $\mu(a_+), \mu(a_-),\mu(b_+)$ and $\mu(b_-)$ have no discontinuous points. 
Then,  for any integer $N\ge 2$, there exists a partial isometry  $u_N \in \cM$ 
 such that  
 \begin{align*}
 &~\quad  \frac{N-1}N     \sigma_{\frac{ N}{N+1}}  \mu\Big(
(\mu(a_+)+\mu(b_-)) \oplus(\mu(a_-)+\mu(b_+))  \Big)  \\
  &\le  \mu(  a u_N   - u_N b )\\
  & \le \frac{N }{N-1 }\mu\Big(
(\mu(a_+)+\mu(b_-)) \oplus(\mu(a_-)+\mu(b_+))  \Big)     . 
 \end{align*}


\end{lemma}
\begin{proof}
We only prove the statement for the case  when $\cM$ is atomless. The case when $\cM$ is atomic follows from the same argument. 
There are $4$ possibilities for the traces of supports $s(a_+)$ and $s(a_-)$:  
\begin{enumerate}
  \item $\tau(s(a_+))=\infty$ and $\tau(s(a_-))<\infty $;
    \item $\tau(s(a_+))<\infty$ and $\tau(s(a_-))=\infty $;
  \item  $\tau(s(a_+))=\tau(s(a_-))=\infty$;
  \item $\tau(s(a_+)), \tau(s(a_-))< \infty$.
\end{enumerate}
Below, we provide a proof for the first case. 
Other cases follow from  similar (and even simpler)
arguments.   
For the operator $b$, there are $6$ possible cases:
	\begin{enumerate}
		\item  $\tau(s(b_+)) \le \tau(s(a_-))<\infty $ and $\tau(s(b_-))=\tau(s(a_+))=\infty $; 
		\item  
		$ \tau(s(a_-))< \tau(s(b_+)) <\infty $
	and $\tau(s(b_-))=  \tau(s(a_+))=\infty $;  
		\item $ \tau(s(a_-)) <\tau(s(b_+))= \infty  $
		and $\tau(s(b_-))=  \tau(s(a_+))=\infty $; 
			\item  $  \tau(s(b_+)) \leq  \tau(s(a_-))
		<\infty $
		and $\tau(s(b_-))< \tau(s(a_+))=\infty $;
		
		\item  $ \tau(s(a_-))<\tau(s(b_+))
		<\infty  $
	 and $\tau(s(b_-))< \tau(s(a_+))=\infty $;

		\item  
		$ \tau(s(a_-)) < \tau(s(b_+))=\infty $
		 and $\tau(s(b_-))< \tau(s(a_+))=\infty $.
	\end{enumerate}
  The proofs for all cases are similar. 
Below, we consider the last case.

%

Since $\cM$ is atomless, it follows from \cite[Lemma 3.7.7]{DPS} that there exists  a projection  $e$ in $\cM $ such that 
$$e^{b_+} \Big(\mu\Big( \tau(s(a_-)) ; b_+ \Big),\infty  \Big)
\le 
e
\le e^{b_+}\Big[\mu\Big( \tau(s(a_-)) ; b_+ \Big),\infty  \Big)
$$
with $\tau(e)=\tau(s(a_-))$. 
Note that $e$ commutes with $b_+$.
%

{\bf Step 1.} 
We consider operators $a_-$ and $b_+e$. 
By Lemma \ref{lemma:decomposition}, 
there exists $0 <\cdots<s_j<s_{j+1} < \cdots < \tau(s(a_-))$, $-\infty <j<\infty $, 
such that $$ \mu(s_j;b_+)  \ge \mu(s_{j+1};b_+)> \frac{N-1}{N}  \mu(s_j;b_+) $$
and 
$$       \mu(s_j;a_-)    \ge \mu(s_{j+1};a_-)> \frac{N-1}{N}  \mu(s_j;a_-)  . $$
 Let $\lambda_j:=\mu(s_j;b_+)$ and
 $\lambda_j':=\mu(s_j;a_-)$. 
 In particular, we have 
 \begin{align}\label{lambdaj}
 	\lambda_{j+1} > \frac{N-1}{N}\lambda_j\quad\mbox{ and } \quad \lambda'_{j+1} > \frac{N-1}{N}\lambda'_j .
 \end{align}

Since $\cM$ is atomless
and $a_-,b_+e$ are  positive,
it follows from
 \cite[Lemma 3.7.7(i)]{DPS} 
 (or \cite[Lemma 3.7.8]{DPS}) 
that for  
$\{\lambda_j\}_j$ and  $\{\lambda_j'\}_j$,  
there exist two  sequences 
$\left\{e_j \right\}_{j}$, $\left\{e_j' \right\}_{j}$ of projections in 
$\cM$ such that 
$$  \cdots<  e_j<e_{j+1}  \cdots \mbox{ and }
\tau(e_j) =s_j, ~\forall j,
$$
$$ \cdots<   e_j'< e_{j+1} '  \cdots  \mbox{ and }
\tau(e_j') =s_j,~ \forall j,
$$
  $$e^{b_+e}(\lambda_j,\infty)\le e_j \le 
  e^{b_+e}[\lambda_j,\infty )  \mbox{
  and  }
 e^{a_-}(\lambda_j',\infty)\le e_j' \le 
  e^{a_-}[\lambda_j',\infty ).$$
In particular, we have $e_j b_+e =b_+ee_j $ and $e_j'a_-=a_-e_j'$. 

Since $\cM$ is a factor, it follows from \cite[Corollary 6.2.6]{KR-II} that  there exists a  partial isometry  $u_1\in \cM $ with 
\begin{align}\label{supp}
\mbox{$u_1 u_1^* = s(a_-)$ and $u_1^*u_1 =e$} 
\end{align} such that 
$$u_1e_ju^*_1 =e_j', ~\forall j. $$
Note that 
we have 
\begin{align} \label{ineqx+}
\begin{split}
u_1 \left(\sum_{j } \lambda_{j+1}  (e_{j+1}-e_j)\right)  u^*_1   &\le u _1 b_+e u^*_1  \\
 & \le u_1 \left( \sum_{j  } \lambda_{j }  (e_{j+1}-e_j) \right ) u_1^*   \\
&\stackrel{\eqref{lambdaj}}{<}\frac{N}{N-1}u_1 \left( \sum_{j  } \lambda_{j+1 }  (e_{j+1}-e_j) \right ) u^*_1 
\end{split}
\end{align}
and 
\begin{align}\label{*}
	u_1\left(\sum_{j } \lambda_{j+1}  (e_{j+1}-e_j)\right)  u_1 ^* +a_-  &\notag  \le u_1 b_+e u^*_1  +a_- 
    \\ & \le u_1 \left( \sum_{j  } \lambda_{j }  (e_{j+1}-e_j) \right ) u_1 ^* +a_- . 
	\end{align}
Consequently, 
\begin{align*}\begin{split}
&~\quad \mu\left(
u_1  b_+e u^*_1  +a_- - \left(
 u _1 \left( \sum_{j} \lambda_{j+1}   (e_{j+1}-e_j) \right)  u^*_1  +a_- 
 \right)
  \right) \\
  &\le \mu\left(
  u _1 \left( \sum_{j} (\lambda_{j}-\lambda_{j+1}  ) (e_{j+1}-e_j) \right) u^*_1 
  \right).
  \end{split}
\end{align*} 
 By  \eqref{lambdaj}, we have $\lambda _j -\lambda_{j+1 }< \frac{1}{N}\lambda_j$, therefore, we obtain 
  \begin{align}\label{1N-1}
  \begin{split}
&  \quad    \mu\left(
u_1  b_+e u^*_1  +a_- - \left(
 u _1 \left( \sum_{j} \lambda_{j+1}   (e_{j+1}-e_j) \right)  u^*_1  +a_- 
 \right)
  \right) \\
   &\, \le  \frac 1N \mu\left(
  u _1\left( \sum_{j} \lambda_{j } (e_{j+1}-e_j) \right)  u^*_1
  \right) \\
  & 
  \stackrel{\eqref{ineqx+}}{\le} \frac 1N \frac {N}{N-1}\mu\left(
  u_1\left(\sum_{j} \lambda_{j+1  } (e_{j+1}-e_j)\right)  u_1^*
  \right)  \\
  &
  \stackrel{\eqref{ineqx+}}{\le}  \frac 1{N-1} \mu\left(
  u_1  b_+e u^* _1
  \right) \\
  &\,= 
   \frac 1{N-1} \mu\left(
  u_1 e  b_+e u^* _1
  \right) \\
  &\stackrel{\eqref{supp}}{=} \frac 1{N-1}  \mu(e b_+e)
\\
  &\, = \frac 1{N-1}  \mu(b_+e)
  \end{split}
  \end{align}
and, in particular, we have 
\begin{align}\label{NN-1mu(b)}
	 \mu\left(
	u _1\left( \sum_{j} \lambda_{j } (e_{j+1}-e_j) \right)  u^*_1
	\right) \stackrel{\eqref{1N-1}}{\leq} \frac {N}{N-1}  \mu(b_+e).
\end{align}
Note that 
\begin{align}\label{orthogonal2}
	&\qquad  \qquad \notag\mu\left(
	u_1 \left( \sum_{j}  \lambda_{j}  (e_{j+1}-e_j) \right)  u^*_1  +a_- 
	\right)  \\ &
	\,\qquad =	
	 \mu\left(
\sum_{j}  \lambda_{j}  (e_{j+1}'-e_j')  +a_-
	\right) 
	\nonumber \\ &
	\,\qquad =	
	\mu\left(
 \sum_{j}  \lambda_{j}  (e_{j+1}'-e_j')  +\sum_{j} a_- (e_{j+1}'-e_j')
	\right) 
	 \\
	 &
	\,\qquad=	\mu\left(
	\sum_{j} \left( \lambda_{j} + a_- \right) (e_{j+1}'-e_j')    
	\right) \nonumber \\
	 &\, \qquad= 	\mu\left(
	 \mathop\oplus_{j} \left( \lambda_{j} + a_- \right) (e_{j+1}'-e_j')    
	 \right)\nonumber \\ &\nonumber
	\,\qquad =\mu\left(\mathop\oplus_{j}	\mu\left(
	  \left( \lambda_{j} + a_- \right) (e_{j+1}'-e_j')   
	 \right) \right) \\
	  &\notag 
	  \stackrel{\mbox{\tiny \cite[Prop.3.2.8]{DPS}}}{=}
     \mu\left(\mathop\oplus_{j} 	 \left(
      \lambda_{j} \chi_{[s_{j}, s_{j+1})} + \mu
	 \left(  a_-   \right)   \chi_{[s_{j}, s_{j+1})} \right)  \right)  .
\end{align}
Since $\{\lambda_j\}_j$ is decreasing, it follows that
\begin{align*}	
	&\notag \qquad \mu\left(
	u_1 \left( \sum_{j}  \lambda_{j}  (e_{j+1}-e_j) \right)  u^*_1  +a_- 
	\right) \\ 
    &\notag   \stackrel{\eqref{orthogonal2}}{=} 
    \mu\left(\mathop\oplus_{j} 	 \left(
      \lambda_{j} \chi_{[s_{j}, s_{j+1})} + \mu
	 \left(  a_-   \right)   \chi_{[s_{j}, s_{j+1})} \right)  \right)  
     \\ 
     &\notag  ~ = 
 \oplus_{j} 	 \left(
      \lambda_{j} \chi_{[s_{j}, s_{j+1})} + \mu
	 \left(  a_-   \right)   \chi_{[s_{j}, s_{j+1})} \right) 
\\ &
	\,\,= \left( \sum_{j} \lambda_{j}\chi_{[s_{j}, s_{j+1})}\right) +	\left( \sum_{j}\mu
	\left(  a_-  \right)  
	\chi_{[s_{j}, s_{j+1})} \right) 	
 \\ &
	\,\,= \mu\left(
	\sum_{j}  \lambda_{j}  (e_{j+1}'-e_j')	\right) +\mu \left(\sum_{j} a_-(e_{j+1}'-e_j')\right)\nonumber \\
	& \nonumber\,\,=\mu\left(
	\sum_{j}  \lambda_{j}  (e_{j+1}'-e_j')	\right) +\mu ( a_-)
	.			
	\end{align*}
Similarly, we have \begin{align}\label{orthogonal}
	& \notag \quad  \mu\left(
	u_1 \Big( \sum_{j}  \lambda_{j+1}  (e_{j+1}-e_j) \Big)  u^*_1  +a_- 
	\right) \\ & =
	\mu\left(
	u_1 \Big( \sum_{j}  \lambda_{j+1}  (e_{j+1}-e_j) \Big)  u^{*}_1 
	\right) +\mu(a_-) .
\end{align}	
	Hence,  we have
\begin{equation}\label{first estimate}
\begin{split}
	\allowdisplaybreaks
 &\frac{N-1}{N}\mu(b_+e)+\mu(a_-)
  \stackrel{\eqref{1N-1}}{=}
  \frac{N-1}{N}\mu(u_1  b_+e u^*_1 )+\mu(a_-)\\&
  \stackrel{\eqref{ineqx+}}{<} 
  \frac{N-1}{N}\frac{N}{N-1}\mu\left(
  u_1 \left( \sum_{j}  \lambda_{j+1}  (e_{j+1}-e_j) \right)  u^*_1 
  \right)+\mu(a_-)
  \\&
\, = 
\mu\left(
 u_1 \left( \sum_{j}  \lambda_{j+1}  (e_{j+1}-e_j) \right)  u^*_1 
 \right) +\mu(a_-)
 \\
 &
 \stackrel{\eqref{orthogonal}}{=}\mu\left(
  u_1 \left( \sum_{j}  \lambda_{j+1}  (e_{j+1}-e_j) \right)  u^*_1  +a_- 
  \right)  \\
  &\stackrel{\eqref{*}}{\leq}
  \mu\left(
 u_1  b_+ e u^* _1 +a_- 
  \right)\\
  &\stackrel{\eqref{*}}{\leq} \mu\left(
  u_1 \left( \sum_{j}  \lambda_j  (e_{j+1}-e_j)\right)  u^*  _1 +a_-
  \right)  \\
  &\stackrel{\eqref{orthogonal2}}{=}
  \mu\left(
  u_1 \left( \sum_{j}  \lambda_{j}  (e_{j+1}-e_j) \right)  u^*_1 
  \right) +\mu(a_-)\\
  & \stackrel{\eqref{NN-1mu(b)}}{\le} \frac{N}{N-1}\mu(b_+e)+\mu(a_-). 
  \end{split}
  \end{equation}
In particular, \eqref{first estimate} implies that
  \begin{equation}\label{eq+}\begin{split}&~\quad 
  		\frac{N-1}{N}\sigma_{\frac{N}{N+1}}\mu(b_+e)+\frac{N-1}{N}\sigma_{\frac{N}{N+1}}\mu(a_-)\\
  		&  \le  \frac{N-1}{N}\mu(b_+e)+\mu(a_-) \\
&  		\leq
  		\mu\left(
  		u_1  b_+ e u^* _1 +a_- 
  		\right)\\
  		 & 
  		 \le \frac{N}{N-1}\mu(b_+e)+\mu(a_-) \\
  		& \leq \frac{N}{N-1}\mu(b_+e)+\frac{N}{N-1}\mu(a_-).	
  	\end{split}
  \end{equation}
 
{\bf Step 2.} Let $e'$ be a projection in $\cM$ with $\tau(e') = \tau(s(b_-))$ and  $$e^{a_+} \Big(\mu\Big( \tau(s(b_-)) ; a_+  \Big),\infty  \Big)
\le 
e' 
\le e^{a_+}\Big[\mu\Big( \tau(s(b_-)) ; a_+ \Big),\infty  \Big). 
$$
In particular, we have $a_+ e' =e'a_+$.  
Now we consider operators $a_+e'$ and $b_-$. 
 Arguing similarly as above, we obtain a partial isometry  $u_2 \in \cM $ with $u_2 u_2^* =e'  $ and $u_2^* u_2 =s(b_-)$ such that 
  \begin{align*}
\mu(a_+ e')+ \frac{N-1}{N} \mu(b_-)
 \le  
  \mu\left(
  a_+ e'  +u_2 b_-u_2  ^*  
  \right)        \le  \mu(a_+e' )+ \frac{N}{N-1} \mu(b_-). 
  \end{align*}
In particular, this shows
  	\begin{equation}\label{eq-}\begin{split}
  		&~\quad 	\frac{N-1}{N}\sigma_{\frac{N}{N+1}}\mu(a_+ e')+\frac{N-1}{N}\sigma_{\frac{N}{N+1}}\mu(b_-)\\
  		&
  			\leq
  			 \mu\left(
  			a_+ e'  +u_2 b_-u_2  ^*  
  			\right)    \\&
  			\leq \frac{N}{N-1}\mu(a_+e')+\frac{N}{N-1}\mu(b_-). 
  		\end{split}
  \end{equation}

{\bf Step 3.} Now, we consider $b_+ ({\bf 1} -e )$ and $a_+ ({\bf 1} -e' )$,  both of which  have infinite supports.  
 Since the algebra $\cM$ is atomless, it follows from 
\cite[Lemma 3.7.8]{DPS} 
 that  there exist a sequence $\left\{f_k \right\}_{k\ge 1 }$ of projections in $\cM$
such that $0=f_0 < f_1<f_2<\cdots$, 
$\tau(f_k)=k$
and 
\begin{align*}
    e^{b_+({\bf 1} -e)  }(\mu(k;b_+ ({\bf 1} -e )),\infty ) 
    &\le 
    f_k 
    \le 
    e^{b_+({\bf 1} -e )}  [\mu(k;b_+ ({\bf 1} -e ) ),\infty ) ,
    \\  &
[b_+({\bf 1} -e ), f_k] =0 ,  
\end{align*}
and 
 a sequence $\left\{f_k ' \right\}_{k\ge 1 }$ of projections in $\cM$
such that $0=f'_0< f_1'<f_2'<\cdots$, 
$\tau(f_k' )=k$
and 
\begin{align*}
e^{a_+({\bf 1} -e')  }(\mu(k;a_+  ({\bf 1} -e' )),\infty ) & \le f_k'  \le e^{a_+({\bf 1} -e' )}[\mu(k;a_+ ({\bf 1} -e' ) ),\infty ) ,   \\  &
[a_+({\bf 1} -e' ), f_k'] =0 . 
\end{align*}

Recall that   $a,b\in S_0(\cM,\tau)$. 
Let    $   \{m_k\}_{k\ge 0}   $  
be a sequence  of positive integers such that    $m_{k+1}\ge N m_k$ ($m_0\ge  N$)
and 
\begin{align}\label{assumption}
\frac1N   \cdot  \mu(k +1   ;a _+ ({\bf 1} -e '))  \ge \mu(m_{k} ;b_+  ({\bf 1} -e ) )  	
\end{align}
for all $k\ge 0 $.


For every $k\ge 0$, 
let    $$g_k:=  f_{m_k+1 }-f_{m_k  }, ~g_k':=  f_{k+1 }'-f_{k  }'  \mbox{ with } \tau(g_k)=\tau(g_k')=1 .$$
Since $\cM$ is a factor, it follows from \cite[Corollary 6.2.6]{KR-II} that  there exists a partial isometry   $v$ in $\cM$ with $ vv^*  =\mathop{\vee}\limits_{k\ge 0} s(g_k)  $ and $  v^* v ={\bf 1}-e  ' $ such that 
 $$ v g_{k}' v^* = g_k  \mbox{ and 
  } v ^* g_k  v  =g_{k} ' .  $$

 By the definition of $g_k$'s, we have 
   \begin{align}\label{a(1-e')}
  	\begin{split}
  	 \mu(k+1; a _+ ({\bf 1} -e'))g_{k} \leq va _+ ({\bf 1} -e')v^*g_{k}
  	\leq \mu(k; a _+ ({\bf 1} -e'))g_{k}
  	\end{split}
  \end{align}
  and
  \begin{align*} 
  	\begin{split}
    b_+	g_k 
  		\leq
  		\mu(m_k; b_+({\bf 1}-e))g_k\stackrel{\eqref{assumption}}{\leq}\frac1N    \mu(k +1  ;a _+ ({\bf 1} -e ')) g_k
  	\end{split}
  \end{align*}
  for all $k\geq0$,  and therefore, 
  \begin{align*} 
  	\begin{split}
  		va _+ ({\bf 1} -e')v^*g_k-b_+ g_k
  	& \,\, \geq   va _+ ({\bf 1} -e')v^*g_k  - \frac1N    \mu(k +1  ;a _+ ({\bf 1} -e ')) g_k   \\ &\stackrel{\eqref{a(1-e')}}{\geq}
  	\frac{N-1}{N}va _+ ({\bf 1} -e')v^*g_{k}
  	, 
  	\end{split}
  \end{align*}
 and  so, 
   \begin{align*}
  	\mu 	\left(va _+ ({\bf 1} -e')v^*g_k-b_+g_k\right)
  	 \geq \frac{N-1}N   \mu\left(va _+ ({\bf 1} -e')v^*g_k\right)
  	.
  \end{align*}
Then, we have 
\begin{align}\label{ineqs(gk)}
	\begin{split}
		\mu 	\left(va _+ ({\bf 1} -e')v^*-b_+ 
     ( 
     \mathop{\vee}_{k\ge 0} s(g_k)
     ) 
        \right)&=	
		\mu 	\Big(\mathop\oplus_{k\geq0}
		\left(va _+ ({\bf 1} -e')v^*g_k-b_+g_k
		\right)\Big)  \\ &
		\geq \frac{N-1}N \mu \Big(\mathop\oplus_{k\geq0}	\mu\left(va _+ ({\bf 1} -e')v^*g_k\right)\Big) 
		\\ &
	\geq \frac{N-1}N   \mu\left(va _+ ({\bf 1} -e')v^*\right).
		\end{split}
\end{align}

  Now,
   we consider
 $$ \mu\Big( b_+({\bf 1}-e-\mathop               {\vee}_{k\ge 0} s(g_k)) \Big)=\mu\left(\Big(\chi_{[0, m_0 )}+\sum _{k\geq 1}\chi_{[m_{k-1}+1, m_k)}\Big)\mu(  b_+({\bf 1} -e) )\right).
  $$ 
    Note that 
 \begin{align}\label{m0}
 	\begin{split}
 \Big(\sigma_{\frac{N+1}{N}}	\left(\chi_{[0, m_0)}\mu( b_+({\bf 1} -e) )
 \right)\Big) (s)& 
 \geq  \Big(\sigma_{1+\frac{1}{m_0}}	\left(\chi_{[0, m_0)}\mu( b_+({\bf 1} -e) )
 \right)\Big) (s)
  \\ &
 	 =\chi_{[0,  m_0+1)}(s)\mu\left( \frac{sm_0}{m_0+1}; b_+({\bf 1} -e)\right) 
 	\\ & \geq	\chi_{[0,  m_0+1)}(s)\mu\left( s; b_+({\bf 1} -e)\right) 	.
 \end{split}
 	\end{align}
On the other hand, we have 
\begin{align*}
	\begin{split}
     &~\quad 
\Big(\sigma_{\frac{N+1}{N}}	\left(\chi_{[m_{k-1}+1, m_k)}\mu(  b_+({\bf 1} -e) ) \right) \Big) (s)\\
&\geq \Big(\sigma_{ \frac{m_k-m_{k-1} }{m_k-m_{k-1}-1}}	\left(\chi_{[m_{k-1}+1, m_k)}\mu(  b_+({\bf 1} -e) ) \right) \Big) (s)
\\ &
=\chi_{\left [(m_{k-1}+1) \frac{m_k-m_{k-1} }{m_k-m_{k-1}-1} , m_k \frac{m_k-m_{k-1} }{m_k-m_{k-1}-1} \right)}(s)
  \\ & 
    \quad \cdot 
    \mu\left( \frac{s(m_k-m_{k-1}-1)}{m_k-m_{k-1}}; b_+({\bf 1} -e)\right) 	
	\\ &
	=\chi_{\left [(m_{k-1}+1) \frac{m_k-m_{k-1} }{m_k-m_{k-1}-1}  , (m_{k-1}+1) \frac{m_k-m_{k-1} }{m_k-m_{k-1}-1} + (m_k-m_{k-1 })\right)}(s)
    \\ & 
    \quad \cdot 
    \mu\left( \frac{s(m_k-m_{k-1}-1)}{m_k-m_{k-1}}; b_+({\bf 1} -e)\right) 
,
\end{split}
\end{align*}
and therefore,
by the definition of generalized singular value functions, we have 
\begin{align*}
	\begin{split}&~\quad 
\mu\Big(s; \sigma_{\frac{N+1}{N}}	\left(\chi_{[m_{k-1}+1, m_k)}\mu(  b_+({\bf 1} -e) ) \right) \Big) \\
&\geq    \chi_{\left [0 , m_k-m_{k-1 }\right)}(s)\mu\left( 
m_{k-1}+1+ 
\frac{s(m_k-m_{k-1}-1)}{m_k-m_{k-1}}; b_+({\bf 1} -e)\right) \\
&\geq    \chi_{\left [0 , m_k-m_{k-1 }\right)}(s)\mu\left( 
m_{k-1}+1+ 
s ; b_+({\bf 1} -e)\right) 
\\
&= \mu\left ( s   ;   \chi_{\left [m_{k-1}  +1  , m_k +1 \right)}\mu\left( s; 
 b_+({\bf 1} -e)\right)   \right) 
\end{split}
\end{align*}
for all $k\geq 1$.
 This together with \eqref{m0} implies that  
 \begin{align}\label{sigma1}
   	\sigma_{\frac{N+1}{N}}\mu\Big( b_+({\bf 1}-e-\mathop{\vee}_{k\ge 0} s(g_k)) \Big)
   	\geq \mu( b_+({\bf 1} -e) ).
   \end{align}
   Hence,   we have 
    \begin{align} \label{secondineq}
    \mu\Big( b_+({\bf 1}-e-\mathop{\vee}_{k\ge 0} s(g_k)) \Big)\stackrel{\eqref{sigma1}}{\geq } \sigma_{\frac{N}{N+1}}\mu( b_+({\bf 1} -e) )\geq \frac{N-1}{N} \sigma_{\frac{N}{N+1}}\mu( b_+({\bf 1} -e) ).
 \end{align}
Now, we  
obtain that
\begin{align*}
&\qquad \quad  \frac{N-1}N   \sigma_{\frac{N}{N+1}} \mu  \Big(\mu(a_+ ({\bf 1}-e'  ) )\oplus  \mu(   b _+({\bf 1}-e )  )  \Big)
\\ & \quad \quad\, =\mu\Big(\frac{N-1}N  \sigma_{\frac{N}{N+1}}\mu\left(a _+ ({\bf 1} -e')\right)\oplus \frac{N-1}N
\sigma_{\frac{N}{N+1}}\mu\left( b _+({\bf 1}-e )\right)\Big)
  \\
 &  \quad 
\stackrel{\eqref{ineqs(gk)}, \eqref{secondineq}}{\leq}
 \mu\left( \mu\Big(
 v (a_+ ({\bf 1}-e'  ))v^*  - b_+
 \Big(\mathop{\vee}_{k\ge 0} s(g_k)\Big)
 \Big)
 \oplus
  \mu\Big(
   b_+(
   {\bf 1}-e-
   \mathop{\vee}_{k\ge 0} s(g_k)
   ) 
   \Big)
   \right) \\
   &
 \quad  \quad\, =
   \mu\left(\Big(
   v (a_+ ({\bf 1}-e'  ))v^*  - b_+\Big(\mathop{\vee}_{k\ge 0} s(g_k)\Big)\Big)
   \oplus
 (-  b_+({\bf 1}-e-\mathop{\vee}_{k\ge 0} s(g_k)))  
   \right)\\ 
   & \quad \quad\, =\mu(v  a_+({\bf 1}-e' ) v^*-b_+({\bf 1}-e ))\\
 &  \stackrel{\mbox{\tiny \cite[Thm.3.4(1)]{HanS}}}{\le}
  \mu\Big( \mu
(a_+ ({\bf 1}-e'  )) \oplus
 \mu( b_+({\bf 1}-e ))\Big)
 \\
& \quad  \quad\,\,\le \frac{N }{N-1 }  \mu\Big(\mu(a_+ ({\bf 1}-e'  )) \oplus  \mu(  b_+({\bf 1}-e )  ) \Big).  
\end{align*}
This together with \eqref{eq+} and \eqref{eq-} shows that 
\begin{align*}
&~\quad \frac{N-1}N    \sigma_{\frac{N }{N+1}} \mu\Big(
(\mu(a_+)+\mu(b_-)) \oplus(\mu(a_-)+\mu(b_+))  \Big) 
 \\
&
 \le  \mu\Big( \mu(-u_1^*  a_- u_1 -b_+ e) \oplus \mu(u_2^*  a_+e' u_2+b_- )\oplus \mu(v  a_+({\bf 1}-e' ) v^*-b_+({\bf 1}-e ) ) \Big)
 \\
&  = \mu\left(  (-u_1^*  a_- u_1 -b_+ e) \oplus ( u_2^*  a_+e' u_2+b_- )\oplus (v  a_+({\bf 1}-e' ) v^*-b_+({\bf 1}-e )) \right)
\\
& 
= 
\mu\left( u_N^*  a u_N   - b  \right) \\
& = \mu\left(   a u_N   - u_N b  \right) 
\\
&\le  \frac{N }{N-1 }\mu\Big(
(\mu(a_+)+\mu(b_-)) \oplus(\mu(a_-)+\mu(b_+))  \Big)  , 
\end{align*}
where $u_N = u_1+u_2 +v^* $. 
\end{proof}


Now, we provide a complete proof for Theorem \ref{main2}.
\begin{proof}
Without loss of generality, we may assume that $\mu(a_+)$, $\mu(a_-)$, $\mu(b_+)$ and $\mu(b_-)$    are continuous on $(0,\infty)$.
Indeed, there exists a commutative atomless von Neumann subalgebra $\cA_1$ of $\cM_{s(a_+)}$ such that $a_+\in \cA_1$ and $S(\cA_1,\tau)$ is (trace-measure preserving) $*$-isomorphic to $S (0,\tau(s(a_+)))$  (see e.g. \cite{CS,CKS}).
For any $\delta>0$, there exists a positive function $f$ which is continuous on $(0,\tau(s(a_+)))$ such that $\norm{f-\mu(a_+) }_E<\delta$. 
Hence, there exists a positive operator $a'$ such that $\norm{a'-a_+  }_{E(\cM, \tau)}  <\delta$. 
Therefore, we may assume that $\mu(a_+)$, $\mu(a_-)$, $\mu(b_+)$ and $\mu(b_-)$    are continuous on $(0,\infty)$.

By Lemma \ref{lemma:approx},   for any $N\ge 2$, there exists a partial isometry  $u_N \in \cM$  
 such that  
\begin{align*}
&~\quad 
\frac{N-1}{N+1}\mu\Big ((\mu(a_+) +\mu(  b_- )) \oplus (\mu(a_- )+\mu( b_+ )) \Big) 
\\
&\stackrel{\eqref{dilationuniform}}{\vartriangleleft}   \frac{N-1}N     \sigma_{\frac{  N}{N+1}}  \mu \Big((\mu(a_+) +\mu(  b_- )) \oplus (\mu(a_- )+\mu( b_+ ))\Big)\\
   &\le  \mu(   a u_N   - u_N b ) \\
   &\le \frac{N}{N-1} \mu\Big ((\mu(a_+) +\mu(  b_- )) \oplus (\mu(a_- )+\mu( b_+ )) \Big) . 
 \end{align*}
Since 
$E(0,\infty)$ is a  symmetrically normed space, it follows from \cite[Corollary 3.4.3]{LSZ} that 
\begin{align*}
  &\frac{N-1}{N+1}\norm{ (\mu(a_+) +\mu(  b_- )) \oplus (\mu(a_- )+\mu( b_+ )) }_{E(0,\infty )}  \\& \le 
 \norm{\mu(a u_N  - u_N  b )}_{E(0,\infty )}\\
& =\norm{\delta_{a,b}(u_N ) }_{E(\cM,\tau)} . 
\end{align*}
Since $N$ is arbitrarily taken, it follows from Theorem \ref{upper estimate} that 
$$\norm{\delta_{a,b} }_{\cM\to E(\cM,\tau)}= \norm{(\mu(a_+) +\mu(  b_- )) \oplus (\mu(a_- )+\mu( b_+ ))   }_{E(0,\infty )}. $$
This completes the proof. 
\end{proof}

\begin{rem}

The condition that $\cM$ is an infinite factor can not be relaxed because
 $\norm{\delta_a}=0$ for any operator  $a$ in an abelian von Neumann algebra $\cM$; 
 the condition that $a\in S_0(\cM,\tau)$ is also necessary, e.g., for   a semifinite infinite factor  $\cM$, ${\bf 1}\notin S_0(\cM,\tau)$ and $\norm{\delta_{\bf 1}}=0$.
\end{rem}

\section{Proof of Theorem \ref{thm2}}

The next corollary follows  immediately  from Theorem \ref{upper estimate}.
\begin{cor}\label{normest}Let $\cM$ be a  von Neumann algebra equipped with a semifinite faithful normal trace $\tau$.
Let $a,b\in S(\cM,\tau)
$ be positive operators and
$x\in\cM$. We have
$$\norm{ax-xb}_{E(\cM,\tau)} \le \norm{x}_\cM \norm{a\oplus b}_{E(\cM\oplus \cM,\tau \oplus \tau)}$$
for any symmetrically normed function space $E(0,\infty)$.
\end{cor}

%

Let $\cM$ be a  von Neumann algebra, and let  $a=a^*\in LS(\cM)$.
The algebra of all self-adjoint elements from the center of the algebra $LS(\cM)$ will be denoted by $Z_h(LS(\cM))$.
  For arbitrary $c,c_1,c_2\in Z_h(LS(\cM))$ such that $c_1\leq c_2$,  the following projections were introduced in  \cite[p.551-552]{BS12b}:
$$e_\cZ^a(-\infty,c):=s((c-a)_+),\ e_\cZ^a(c,+\infty):=s((a-c)_+),$$
$$e_\cZ^a(-\infty,c]:=\textbf{1}-e_\cZ^a(c,+\infty),\ e_\cZ^a[c,+\infty):=\textbf{1}-e_\cZ^a(-\infty,c),$$
$$e_\cZ^a[c_1,c_2):=e_\cZ^a[c_1,+\infty)e_\cZ^a(-\infty,c_2),\ e_\cZ^a(c_1,c_2]:=e_\cZ^a(c_1,+\infty)e_\cZ^a(-\infty,c_2],$$
$$e_\cZ^a[c_1,c_2]:=e_\cZ^a[c_1,+\infty)e_\cZ^a(-\infty,c_2],\ e_\cZ^a(c_1,c_2):=e_\cZ^a(c_1,+\infty)e_\cZ^a(-\infty,c_2),$$
$$e_\cZ^a\{c\}:=e_\cZ^a[c,c].$$ The following sets were also defined there:
$$\Lambda_-=\{c\in Z_h(LS(\cM)): pe_\cZ^a(-\infty,c)\prec pe_\cZ^a(c,+\infty),\ \forall p\in P( Z(\cM)),\ p>0\},$$
$$\Lambda_+=\{c\in Z_h(LS(\cM)): pe_\cZ^a(c,+\infty)\prec pe_\cZ^a(-\infty,c),\ \forall p\in P(Z(\cM)),\ p>0\}.$$
We define an
 element\cite[p.554]{BS12b} (see also \cite{BCLSZ})
 \begin{align*}
 c_a:=\bigvee\Lambda_-\in Z_h (LS(\cM)).
\end{align*}
Note that, if $\cM$ is properly infinite and $a\in S_0(\cM,\tau)$, then $c_a=0$\cite[Corollary B.2]{BHS}. 


The following proposition is an analogue of \cite[Lemma 3.3.7]{LSZ} in terms of uniform submajorization. 
\begin{proposition}\label{directsum sub}
Let $0\le x,y\in L_1(\cM,\tau)+\cM$. We have 
$$\mu(x\oplus y)\vartriangleleft \mu(x+y).$$
\end{proposition}
\begin{proof}
 We have
$$
x+y\stackrel{\text{\cite[Thm.3.3.3]{LSZ}}}{\prec\prec}\mu(x)+\mu(y)
\stackrel{\text{ \cite[Thm.3.3.4]{LSZ}}} {\prec\prec} 2\sigma_{1/2}\mu(x\oplus y).$$
In other words, 
$$
\int_0^a\mu(s;x+y)ds\leq \int_0^{2a}\mu(s;x\oplus y)ds,~a>0.$$
On the other hand, by \cite[Lemma 3.3.7]{LSZ}, we have
$\int_0^b\mu(s;x\oplus y)ds\leq \int_0^b\mu(s;x+y)ds,~b>0.$
Subtracting these inequalities, we obtain
$$\int_{2a}^b\mu(s;x\oplus y)ds \leq \int_a^b\mu(s;x+y)ds.$$
The proof is complete. 
\end{proof}
Note that the above proposition fails for infinite direct sum, that is, 
there are $0\le x_1,x_2 \cdots \in L_1(\cM,\tau)$ such that 
$\mu(\oplus_{i=1}^\infty x_i)\not\vartriangleleft \mu(\sum _{i=1}^\infty x_i )$ (for example, let 
$x_i=2^{-i}\chi_{[0,1)} \in L_1(0,\infty)$.
 This example was essentially given in \cite[p.120]{Kalton_S}).

The following lemma is an analogue of   the main result of \cite{BHS} (in the special case of $E(0,\infty)=L_1(0,\infty)$) for positive operators. 
\begin{cor}
Let $E(0,\infty)$ be a  symmetrically normed function space.
Let
 $\cM$ be a   von Neumann algebra equipped with a semifinite faithful normal trace and let $a \in E(\cM,\tau)$ be a positive operator.
Then,
\begin{align}\label{===}
\norm{a\oplus a}_{E(\cM\oplus \cM,\tau \oplus \tau)}\ge \norm{\delta_{a } }_{\cM\to  E(\cM,\tau)}  \ge \norm{(a-c_a)\oplus (a-c_a) }_{E(\cM\oplus \cM,\tau \oplus \tau)}.
\end{align}

\end{cor}
\begin{proof}
   By \cite[Corollary B.3]{BHS}, for any $\varepsilon>0$,
   there exists $u_\varepsilon=u_\varepsilon^* \in U(\cM )$ such that  
   $$|\delta_{a }(u_\varepsilon)| \ge (1-\varepsilon)\left( |a-c_a |  + u_\varepsilon |a- c_a |u_\varepsilon \right).$$
In particular, we have $a-c_a\in E(\cM,\tau)$. 
   By Proposition \ref{directsum sub}, we have
   $$\mu((a-c_a) \oplus (a-c_a)) \vartriangleleft \left( |a-c_a|  + u_\varepsilon |a-c_a|u_\varepsilon \right).$$
   Hence, we have 
    $$(1-\varepsilon) \mu((a-c_a) \oplus (a-c_a)    )  \vartriangleleft \mu\left( \left| \delta_{a-c_a  }(u_\varepsilon) \right|\right)
.   
$$
  Since  $a$ is a positive operator in $E(\cM,\tau)$ and $E(0,\infty)$ is a  symmetrically normed  space, it follows that 
\begin{align*}
\norm{ a \oplus a  }_{E(\cM\oplus \cM,\tau \oplus \tau)} 
& ~ \stackrel{\rm Cor. \ref{normest}}{ \ge}  \norm{\delta_{a}}_{\cM \to E(\cM,\tau)}\\
&~\quad =~ \norm{\delta_{a-c_a }}_{\cM \to E(\cM,\tau)}\\
&~\quad \ge~ (1-\varepsilon) 
 \norm{\mu((a-c_a) \oplus (a-c_a))}_{E(0,\infty )} \\
 &  ~ \quad = ~ (1-\varepsilon) \norm{(a-c_a)\oplus (a-c_a)}_{E(\cM\oplus \cM,\tau \oplus \tau)}
.
\end{align*}
This completes the proof.
\end{proof}

Even though the above result holds for general semifinite von Neumann algebras, one can not apply it to a generalized derivation.

We begin the proof of Theorem~\ref{thm2} by establishing several auxiliary
lemmas.  The first one is a fundamental result for comparison of projections. 

\begin{lemma}\cite[Lemma 21(iv)]{BS12b}\label{lemma equ} 
For a properly infinite projection $p$ in a von Neumann algebra and a finite projection $r\le p$, we have
$$p\sim p-r. $$
\end{lemma}

We then derive the following consequence. 

\begin{lemma}
\label{lem:cofinite-dominates}
Suppose that $\cM$ is  properly infinite.  If $g,e\in\cP(\cM)$
and ${\bf 1}-e$ is finite, then $g\preccurlyeq e$.
\end{lemma}
\begin{proof}

    By Lemma \ref{lemma equ}, we have 
    $$g\le  {\bf 1} \sim {\bf 1}-({\bf 1}-e)=e.$$ 
    The proof is complete (see e.g. \cite[Proposition 1.14.6]{DPS}). 
\end{proof}


We include the following fact  for later reference. 

\begin{lemma}\cite[Lemma 35]{BS12b} \label{lem:selection} 
Let $\cM$ be a von Neumann algebra, let self-adjoint operator $a\in LS(\cM)$ and let $p,q\in\cP(\cM)$ satisfy $p\succcurlyeq  q$.  Assume that one of the
following conditions holds:
\begin{enumerate}
\item
$q$ is finite, and there exists a non-decreasing  sequence $(p_n)_{n\geq1}$ of
finite projections in $\cM$ such that
\[
p_n\uparrow p
\qquad\text{and}\qquad
ap_n=p_na, 
\quad n\geq1;
\]

\item 
$q$ is properly infinite and
$ 
ap=pa\in\cM.
$
\end{enumerate}
Then there exists a projection $q_1\leq p$ such that
\[
q_1\sim q
\qquad\text{and}\qquad
aq_1=q_1a.
\]
\end{lemma}

Now, we are ready to present the proof of Theorem \ref{thm2}.

\begin{proof}[Proof of Theorem \ref{thm2}]
Fix 
	$0<\varepsilon<1$.  
	We inductively construct non-increasing sequences of projections
	$(r_n^a)_{n\geq0}$ and $(r_n^b)_{n\geq0}$ in $\cM$, starting with
	$r_0^a=r_0^b={\bf 1}$, such that 
	\begin{align}\label{rn1bb=brn1b}
		r_n^a a=a r_n^a,  \quad r_n^b b= b r_n^b, \quad \tau({\bf 1}-r_n^a)<\infty \, 
		\text{ and }  \, \tau({\bf 1}-r_n^b)<\infty . 
	\end{align}

	Suppose that the non-zero projections  $r_{n-1}^a,r_{n-1}^b \in \cM$ have been constructed and  they satisfy \eqref{rn1bb=brn1b}. 
We recall that	two  self-adjoint
	operators $x, y\in S(\cM, \tau)$ commute  if and only if $xe^y (\delta) = e^y (\delta)  x$  for all Borel sets  $\delta\subseteq\mathbb R$ 
	\cite[Proposition 2.2.22]{DPS}. 
	Define  
	\begin{align}\label{eq:Pn}
		P_n:=r_{n-1}^a e^a(2^{-n},\infty), 
		\quad 
			L_n^b:=r_{n-1}^b e^b[0, 2^{-n} \varepsilon].  
	\end{align} 
	By the fact that $	a r_{n-1}^a \stackrel{\eqref{rn1bb=brn1b}}{=} 
	r_{n-1}^a a$, we deduce that   $P_n$ is a 
 projection in $\cM$ and 
	\begin{align} \label{aPncommute}
	&\nonumber 	aP_n
	 \,\,\, \stackrel{\eqref{rn1bb=brn1b}}{=} 
		P_n a \stackrel{\eqref{rn1bb=brn1b}}{=} 
		 	r_{n-1}^a (a e^a(2^{-n},\infty)) r_{n-1}^a
		\\ &	\stackrel{\text{ \cite[Prop.2.2.24]{DPS}}}{\ge}  r_{n-1}^a\, 
		\left(
		 2^{-n} \cdot  e^a(2^{-n},\infty)
		\right) 
		\, r_{n-1}^a   
		\stackrel{\eqref{eq:Pn}}{=}  2^{-n} P_n,
	\end{align}
 Similarly, we know that  $L_n^b$ is a
 projection in $\cM$ and 
	\begin{align}
 	\label{bLnbineq}
		b L_n^b
	 	=L_n^b b \le 2^{-n}  \varepsilon L_n^b.
	\end{align}
	Note that 
	\begin{align*}
		{\bf 1}-L_n^b &\nonumber  =({\bf 1}-r_{n-1}^b)
		+r_{n-1}^b-L_n^b 	\stackrel{\eqref{eq:Pn}}{=} 
		({\bf 1}-r_{n-1}^b)
		+r_{n-1}^b-
		r_{n-1}^b e^b[0, 2^{-n}\varepsilon ] 
		\\ & 
		=({\bf 1}-r_{n-1}^b)
		+r_{n-1}^b e^b(2^{-n} \varepsilon,\infty).
	\end{align*} 
	By the fact that $0\leq b\in S_0(\mathcal M, \tau)$, 
	 we deduce that $\tau(r_{n-1}^b e^b(2^{-n} \varepsilon,\infty)) \leq  
	 \tau(e^b(2^{-n} \varepsilon,\infty) )<\infty$.  It follows from 
     $ \tau({\bf 1}-r_{n-1}^b)\stackrel{\eqref{rn1bb=brn1b}}{<}\infty$   that 
	$\tau({\bf 1}-L_n^b )<\infty$.   
	Since $\cM$ is properly infinite, it follows from  Lemma~\ref{lem:cofinite-dominates} that 
	$P_n\preccurlyeq L_n^b$. 
		Since $0\leq a\in S_0(\mathcal M, \tau)$, it follows that 
	$\tau(P_n)
	\stackrel{\eqref{eq:Pn}}{\leq }
	 \tau ( e^a(2^{-n},\infty)) <\infty$.  Lemma~\ref{lem:selection} yields 
	\begin{equation}\label{eq:Rn}
		\mbox{a projection
			$R_n\leq L_n^b 	\stackrel{\eqref{eq:Pn}}{\leq } r_{n-1}^b$  such that~}	R_n\sim P_n  ~ \text{and}~ R_nb=bR_n.
	\end{equation}
In particular, 
\begin{align}\label{(bRn)etaaPn}
	bR_n
 \stackrel{\eqref{eq:Rn}}{=}
b  (R_n L_n^b R_n)
	\stackrel{ \eqref{eq:Rn}}
	{=} 
R_n (b    L_n^b) R_n \stackrel{\eqref{bLnbineq}}
	{\leq} 
	2^{-n}  \varepsilon  \cdot  R_n  L_n^b R_n
	\stackrel{\eqref{eq:Rn}}{=}
	 2^{-n}  \varepsilon  R_n 
.  
\end{align}
	
	We next interchange the roles of $a$ and $b$ in the preceding matching 
	procedure. 
	Define
	\begin{align}\label{eq:Qn}
		Q_n:=\big(r_{n-1}^b-R_n
		\big) 
		e^b( 2^{-n},\infty), \quad 
		 L_n^a:=
		\big(
		r_{n-1}^a-P_n
		\big) e^a[0, 2^{-n} \varepsilon]. 
	\end{align}
	The same argument shows that  $Q_n$ and $L_n^a$ are
	projections in $\cM$, 
	$\tau(Q_n)<\infty$, $\tau({\bf 1}-L_n^a)<\infty$, \begin{align}\label{bQngeq2{-n}Q_n}
		bQ_n=Q_nb\geq  2^{-n} Q_n, \quad 
        aL_n^a=L_n^a a \leq 2^{-n} \varepsilon L_n^a.
	\end{align} Lemmas~\ref{lem:cofinite-dominates}  and \ref{lem:selection} therefore yield 	
	\begin{equation}\label{eq:Sn}
		\mbox{a projection $S_n\leq L_n^a 
			\stackrel{\eqref{eq:Qn}}{\leq}
				r_{n-1}^a-P_n$ with\,}	 S_n\sim Q_n  \,\text{and}\,  S_na=aS_n \stackrel{\eqref{bQngeq2{-n}Q_n}}{\leq} 	 2^{-n} \varepsilon  S_n.
	\end{equation}

	Finally, set 
	\begin{align}\label{eq:residuals}
		r_n^a & \nonumber :=
		r_{n-1}^a-P_n-S_n
		\stackrel{\eqref{eq:Pn}}{=}
			r_{n-1}^a 
	 e^a[0, 2^{-n}] 
		-S_n 
		 ={\bf 1}-\sum_{k=1}^{n}(P_k+S_k),\qquad
		\\ r_n^b &:=
		r_{n-1}^b-Q_n-R_n
		={\bf 1}-\sum_{k=1}^{n}(Q_k+R_k).
	\end{align}
	Then,   since $a$ commutes with $r_{n-1}^a, P_n , S_n$ (see \eqref{rn1bb=brn1b}, \eqref{aPncommute}, \eqref{eq:Sn}), it follows that 
	\begin{align*}
		r_n^a a = 
		a r_n^a  .
	\end{align*} 
    By  \eqref{rn1bb=brn1b}, \eqref{eq:Pn} and  \eqref{bQngeq2{-n}Q_n}, 
	$\tau({\bf 1}-  r_n^a) =\tau({\bf 1}-r_{n-1}^a)+\tau(P_n)+\tau(S_n) \stackrel{}{<} \infty$.  
	Similarly, we obtain that  
	\begin{align*}
		r_n^b b= b r_n^b  \,  
		\text{ and }  \, \tau({\bf 1} -r_n^b)<\infty . 
	\end{align*} 
	
		By construction, for every $n\geq1$, 
		\begin{align}\label{P_nS_n=0}
				P_nS_n
		\stackrel{\eqref{eq:Sn}}{=}
		P_n
		\big( 
		(r_{n-1}^a-P_n)S_n
		\big)
		\stackrel{\eqref{eq:Pn}}{=}
		(P_n-P_n)S_n=0,
			\end{align}  and similarly,  $R_nQ_n=0$. 
		Moreover, 
			each of the sequences
	$(P_n)_{n\geq1}$, $(S_n)_{n\geq1}$,
	$(R_n)_{n\geq1}$, and $(Q_n)_{n\geq1}$
	consists of mutually orthogonal projections in $\cM$. 
	Indeed, let $m>k$.  
	Since 
	$(r_n^a)_{n\geq0}$  is a non-increasing sequence,  
	it follows that 
\begin{align}\label{Pmleqrm-1a}
	P_m,\, S_m
	\stackrel{\eqref{eq:Pn}, \eqref{eq:Sn}}{\leq} 
	r_{m-1}^a\leq r_k^a \stackrel{ \eqref{eq:residuals}}{=} r_{k-1}^a-P_k-S_k
\end{align}
	and so,  \begin{align*}
		P_kP_m
	& \,\,\,\,\, \stackrel{ \eqref{Pmleqrm-1a}}{=} 
	P_k   
	\big( (  r_{k-1}^a-P_k-S_k ) P_m 
    \big)=
(	P_k   
	  r_{k-1}^a- P_k-P_k S_k) P_m \\ & 
\stackrel{\eqref{P_nS_n=0}, \eqref{Pmleqrm-1a}}{=} 
(	P_k   
	 - P_k) P_m=
0	.
\end{align*}
	Similarly, we obtain  $S_k S_m=R_k R_m=Q_k Q_m=0$. 
By orthogonality,  $r_n^a$ and $r_n^b$ are projections in $\cM$.  
	This completes the induction.

	Define  projections 
	\begin{align} \label{rinftyarinftyb}
		r_\infty^a &\nonumber :={\bf 1}-\sum_{n\ge1}(P_n+S_n) \leq 
		{\bf 1}-\sum_{k=1}^{n} (P_k+S_k)
		\stackrel{ \eqref{eq:residuals}}{=}
		r_n^a, \quad \forall n\geq 1 \qquad
		\\ r_\infty^b &  :={\bf 1}-\sum_{n\ge1}(R_n+Q_n)
		\stackrel{\eqref{eq:residuals}}{\leq} r_n^b, \quad  \forall n\geq 1 , 
	\end{align}
	where the sums converge in the strong operator topology.  	By \eqref{aPncommute} and \eqref{eq:Sn}, we know that 
   $r_\infty^a a=a r_\infty^a$.
	Since $ 2^{-n}\downarrow0$,  it follows that   $$r_\infty^a
	\stackrel{\eqref{rinftyarinftyb}}{\leq}
	\bigwedge_{n\geq1} r_n^a
	\stackrel{\eqref{eq:residuals}}{\leq}
	\bigwedge_{n\geq1}e^a[0, 2^{-n}]
	=e^a\{0\}, $$ and thus,   
	\begin{equation}\label{eq:zero-remainders}
		0\leq a  r_\infty^a=a^{\frac{1}{2}} r_\infty^a a^{\frac{1}{2}} \leq    a^{\frac{1}{2}} e^a\{0\} a^{\frac{1}{2}}=a e^a\{0\} =0. 
	\end{equation} 
	Similarly, we have $r_\infty^b b=b r_\infty^b$, $ r_\infty^b\leq e^b\{0\}$, and so, 
	\begin{equation}\label{eq:bzero-remainders}
	br_\infty^b=0. 
	\end{equation} 
	
	Since for each fixed $n\geq 1$, $
	R_n
	\stackrel{\eqref{eq:Rn}}{\sim} P_n $ and $
	S_n \stackrel{\eqref{eq:Sn}}{\sim}  Q_n $, it follows that 
	there exist partial isometries $v_n,w_n\in\mathcal M$ satisfying
	\begin{align}\label{vnwnRPQS}
		v_n^*v_n=R_n,\quad v_nv_n^*=P_n,\quad
		w_n^*w_n=Q_n,\quad w_nw_n^*=S_n.
	\end{align}	
	A straightforward verification shows   that
	\begin{equation*}
		v:=\sum_{n\geq1}(v_n+w_n)\qquad \text{(s.o.t)}
	\end{equation*}
	is a partial isometry in $\mathcal M$ with   
	\begin{align}\label{vdefPQ}
		Q:= v^*v=\sum_{n\geq1}(R_n+Q_n),\qquad
		P:= vv^*=\sum_{n\geq1}(P_n+S_n).
	\end{align}
	In particular, for each $n\geq 1$, 
	\begin{align}\label{(vbQv^*)P_k}
		(vbQv^*)P_n
		&\nonumber \stackrel{\eqref{vnwnRPQS}}{=}
		vbQv_n^*
		\stackrel{\eqref{vnwnRPQS}}{=}
		vbQR_nv_n^*
		\stackrel{\eqref{vdefPQ}}{=} vbR_nv_n^*
		\stackrel{\eqref{eq:Rn}}{=}
		v(R_nbR_n)v_n^*
		\\
		& \nonumber  \stackrel{\eqref{vnwnRPQS}}{=}
		v_nbR_nv_n^*
		\stackrel{\eqref{vnwnRPQS}}{=}
		v_nbR_nv^*
		\stackrel{\eqref{vdefPQ}}{=} 
		v_nbR_nQv^*
		\stackrel{\eqref{eq:Rn}}{=}
		 v_n R_nb Qv^*
		\\
		&  \stackrel{\eqref{vnwnRPQS}}{=}
		v_nbQv^*
		\stackrel{\eqref{vnwnRPQS}}{=}
		P_n(vbQv^*)
	\end{align}
	and 
	\begin{align}\label{S_k(vbQv^*)}
		S_n(vbQv^*)= (vbQv^*)  S_n=w_nbQ_nw_n^*. 
	\end{align}

	Define 
	\begin{align}\label{HK}
		0\leq H& \nonumber:=\sum_{n\geq1}
		\bigl(aP_n+w_nbQ_nw_n^*\bigr),\\
		0\leq	K&:=\sum_{n\geq1}
		\bigl(v_nbR_nv_n^*+aS_n\bigr).
	\end{align}
By the fact  that
\begin{align}\label{v_n(bR_n)v_n^* eta a P_n}
	v_n(bR_n)v_n^* 
 \stackrel{\eqref{(bRn)etaaPn}}{\leq}
	 2^{-n} \varepsilon  \cdot   v_n R_n v_n^*
  \stackrel{ \eqref{vnwnRPQS}}
	{  =}
	 2^{-n} \varepsilon P_n  
	\stackrel{ \eqref{aPncommute}}
	{  \leq } 
	\varepsilon a P_n 
\end{align}	
and 
	\begin{align}\label{w_n(bQ_n)w_n^*geq}
	aS_n
	\stackrel{\eqref{eq:Sn}}{\leq}
	 2^{-n} \varepsilon S_n
	\stackrel{ \eqref{vnwnRPQS}}
	{  =}
	2^{-n}  \varepsilon \cdot  w_nQ_nw_n^*   
	\stackrel{\eqref{bQngeq2{-n}Q_n}}{\leq} 
		\varepsilon
		 w_n(bQ_n)w_n^*, 
\end{align}	
we deduce that 
	\begin{align}\label{HvarepsilonK}
		 K 	\stackrel{ \eqref{HK}}
		 {  =}
		 \sum_{n\geq1}
		\bigl(v_nbR_nv_n^*+aS_n\bigr)
			\stackrel{ \eqref {v_n(bR_n)v_n^* eta a P_n}, \eqref{w_n(bQ_n)w_n^*geq}
		}
		{  \leq}
		\sum_{n\geq1}
		\bigl(\varepsilon aP_n+\varepsilon w_nbQ_nw_n^* \bigr)
	\stackrel{ \eqref{HK}}
		{  =}\varepsilon H.
	\end{align}

	Since for each $k\geq 1$, 
	\begin{align}\label{P_k	(aP-vbQv^*)geq0}
	P_k	(aP-vbQv^*)  &\nonumber \stackrel{\eqref{(vbQv^*)P_k}}{=}  (aP-vbQv^*) P_k  
 \stackrel{\eqref{(vbQv^*)P_k}}{=}
		aP_k
		-  v_k bR_k v_k^* 
		\stackrel{ \eqref {v_n(bR_n)v_n^* eta a P_n}
		}
		{  \geq} 0,  \\ 
		S_k (aP-vbQv^*) & 
		\stackrel{\eqref{S_k(vbQv^*)}}{=}
		(aP-vbQv^*) S_k
	\stackrel{\eqref{S_k(vbQv^*)}}{=}
	aS_k-  w_k bQ_kw_k^* \stackrel{ \eqref {w_n(bQ_n)w_n^*geq}
		}
		{  \leq}
		0,
	\end{align} 
	and 
	\begin{align}\label{(aP-vbQv*)(1-P)=0}
			(aP-vbQv^*) 	({\bf 1}-P)=aP-vbQv^*-
			aP+vbQv^*P
			\stackrel{\eqref{vdefPQ}}{=}
			0, 
	\end{align} 
	it follows from the  orthogonality of projections  that  
	\begin{align}\label{eq:abs-delta}
		|aP-vbQv^*| 	& \nonumber  \,\,\,\,\, \stackrel{\eqref{(aP-vbQv*)(1-P)=0}}{=}
			|(aP-vbQv^*)P|
		\\ & \nonumber  \stackrel{\eqref{vdefPQ}, \eqref{P_k	(aP-vbQv^*)geq0}}{=}		\sum_{n\geq1}  
		(aP-vbQv^*) P_n
		- \sum_{n\geq1}   
		(aP-vbQv^*) S_n
		\\ & \nonumber  \,\,\,\,\, 
		\stackrel{ \eqref{P_k	(aP-vbQv^*)geq0}}{=}
		\sum_{n\geq1}  \big(  aP_n
		-  v_n bR_n v_n^*  \big) +
		\sum_{n\geq1}  
		\big(  w_n bQ_n w_n^* -
		aS_n 
		\big)
		\\ &\nonumber \,\,\,\,\,\,\,
		=\sum_{n\geq1} \big(  aP_n
		+w_n bQ_n w_n^*  \big) -
		\sum_{n\geq1}  \big( aS_n  +  v_n bR_n v_n^* 
		\big)
		\\ &\,\,\,\,\,\, \stackrel{ \eqref{HK}}
		{  =} H-K.
	\end{align}
	
		
	By the fact that 
$a({\bf 1}- P)\stackrel{\eqref{rinftyarinftyb}, \eqref{vdefPQ}}{=}
ar_\infty^a 
\stackrel{\eqref{eq:zero-remainders}}{=}
0$ and $b({\bf 1}- Q)\stackrel{\eqref{rinftyarinftyb}, \eqref{eq:bzero-remainders},  \eqref{vdefPQ}}{=}0$, 
we deduce that 
$a=aP$ and $b=bQ$. 
Since    
\begin{align}\label{av-vb=(aP-vbQv^*)v}
	av-vb=(aP)v-v(bQ)
	\stackrel{\eqref{vdefPQ}}{=} (aP-vbQv^*)v, 
\end{align}	
it follows 
	that  \begin{align}\label{mu(av-vb)=mu(H-K)}
		\mu(av-vb)
		&\nonumber \stackrel{\eqref{av-vb=(aP-vbQv^*)v}}{=}
		\mu((aP-vbQv^*)v)
		=
		\mu(aP-vbQv^*)
		\\ & 
		\,\, = \mu(|aP-vbQv^*|)
		\stackrel{\eqref{eq:abs-delta}}{=}\mu(H-K).
	\end{align}

	On the other hand, 
	\begin{align}\label{eq:equimeasurabledecom}
		\mu(H\oplus K) 
		&\nonumber  \stackrel{\eqref{HK}}{=}\mu\left( 
		\oplus_{n\geq 1}  \left( aP_n\oplus w_nbQ_nw_n^* 
		\oplus 
		v_nbR_nv_n^*
		\oplus aS_n
		\right)
		\right)\\ & 
		\,\, =\mu\left( 
		\oplus_{n\geq 1}  \left( \mu(aP_n)  \oplus  \mu (w_nbQ_nw_n^*) 
		\oplus
		\mu
		( v_nbR_nv_n^*)
		\oplus \mu (aS_n)
		\right)
		\right)
		.
	\end{align}
	Since  $\mu( v_nbR_nv_n^*)\stackrel{\text{\cite[Prop.3.2.7]{DPS}}}{\leq} \mu(bR_n )$
	and $v_n^*(v_nbR_nv_n^*)v_n \stackrel{\eqref{vnwnRPQS}}{=} R_n b R_n\stackrel{\eqref{eq:Rn}}{=}bR_n$, 
	it follows that 
	$$\mu( v_nbR_nv_n^*)\stackrel{\text{\cite[Prop.3.2.7]{DPS}}}{\geq} \mu(bR_n ), \quad 
	\text{ and so},  \quad  
	\mu( v_nbR_nv_n^*)=\mu(bR_n ).$$ 
	Similarly, we obtain 
	$\mu( w_nbQ_nw_n^* ) =\mu(bQ_n)$. 
	This means  that 
	\begin{align}\label{eq:equimeasurable}
		\mu(H\oplus K) 
		&\nonumber  \,\,\,\,\,\stackrel{\eqref{eq:equimeasurabledecom}}{=}
		\mu\left( 
		\oplus_{n\geq 1}  \left( \mu(aP_n)  \oplus  \mu (w_nbQ_nw_n^*) 
		\oplus
		\mu
		( v_nbR_nv_n^*)
		\oplus \mu (aS_n)
		\right)
		\right)
		\\ & \nonumber
		\quad\,	=\mu\left( 
		\oplus_{n\geq 1}  \left( \mu(aP_n)  \oplus  \mu (bQ_n) 
		\oplus
		\mu
		(bR_n)
		\oplus \mu (aS_n)
		\right)
		\right)
		\\ & \nonumber 
		\stackrel{\eqref{eq:zero-remainders}, \eqref{eq:bzero-remainders}}{=}
		\mu\big( 
		\oplus_{n\geq 1}  \left( \mu(aP_n)  \oplus  \mu (bQ_n) 
		 \oplus
		\mu
		(bR_n)
		\oplus \mu (aS_n) 
		\right)
		\\ &\nonumber \qquad \,\,\,\,     \oplus 
		\mu(a r^a_{\infty} )
		\oplus 
		\mu(b r^b_{\infty} )
		\big)
		\\ & \nonumber \quad\, =
		\mu\left( 
		\oplus_{n\geq 1}  \left( aP_n \oplus  aS_n  
		\oplus 
		bR_n
		\oplus bQ_n
		\right)
		\oplus a r^a_{\infty}
		\oplus b r^b_{\infty}
		\right)
		\\ &\,\,\,\,\,  \stackrel{\eqref{rinftyarinftyb}}{=}  \mu(a\oplus b). 
	\end{align}
	
	Recall that $ 0\leq K
	\stackrel{ \eqref{HvarepsilonK}}
	{  \leq }
	\varepsilon H$. 
	Therefore,  
	\begin{align}\label{Hleqfrac11-eta(H-K)}
		0\leq H\leq \frac{1}{1-\varepsilon} (H-K).
	\end{align}
	By the symmetry of $E(\mathcal M, \tau)$, 
	we  then obtain 
	\begin{align*}
		\|a\oplus b\|_{ 
			E(\mathcal M\oplus\mathcal M,\tau\oplus\tau)} 
		&\stackrel{\eqref{eq:equimeasurable}}{=}
		\|H\oplus K\|_{ 
			E(\mathcal M\oplus\mathcal M,\tau\oplus\tau)} 
		\\
		&\,\, \leq \|H\oplus0\|_{ 
			E(\mathcal M\oplus\mathcal M,\tau\oplus\tau)} 
		+\|0\oplus K\|_{ 
			E(\mathcal M\oplus\mathcal M,\tau\oplus\tau)} \\
		&\,\, =\|H\|_{E(\mathcal M,\tau)}+\|K\|_{E(\mathcal M,\tau)}\\
		& \stackrel{ \eqref{HvarepsilonK}}
		{  \leq }
		(1+\varepsilon)\|H\|_{E(\mathcal M,\tau)} 
		\stackrel{\eqref{Hleqfrac11-eta(H-K)}}{\leq}
		\frac{1+\varepsilon}{1-\varepsilon}
		\|H-K\|_{E(\mathcal M,\tau)}\\
		&\stackrel{\eqref{mu(av-vb)=mu(H-K)}}{=}
		\frac{1+\varepsilon}{1-\varepsilon}
		\|av-vb\|_{E(\mathcal M,\tau)}
		\leq \frac{1+\varepsilon}{1-\varepsilon} \norm{\delta_{a,b}}_{\cM\to E(\cM,\tau)} \norm{v}_{\infty}
		\\ &\,\,  \leq \frac{1+\varepsilon}{1-\varepsilon} \norm{\delta_{a,b}}_{\cM\to E(\cM,\tau)}.
	\end{align*} 
	Since $0<\varepsilon<1$ is arbitrarily taken, 
	we have $$\norm{\delta_{a,b}}_{\cM\to E(\cM,\tau)}
	\ge \norm{a\oplus b}_{E(\cM\oplus\cM,\tau\oplus \tau)}.$$

	By Corollary \ref{normest} above, we complete the proof. 
\end{proof}

\begin{remark}

Observe that, if $a$ is not positive, then   inequality  \eqref{===}
may be false.
For example, let $\cM =\mathbb{M}_2$ (and $\tau$ be a standard trace on $\mathbb{M}_2$) and let $E(\cM,\tau)=L_p(\cM,\tau)$, where  $1<p<\infty$.
Let $a= \left(
          \begin{array}{cc}
            -\frac12 & 0 \\
            0 & \frac12 \\
          \end{array}
        \right)
$
and
$x= \left(
          \begin{array}{cc}
            0 & 1 \\
            1 & 0 \\
          \end{array}
        \right).
$
We have
$$[a,x]= \left(
          \begin{array}{cc}
            0 & -\frac{1}{2} \\
             \frac{1}{2}& 0 \\
          \end{array}
        \right)-\left(
          \begin{array}{cc}
            0 &  \frac{1}{2} \\
             -\frac{1}{2}& 0 \\
          \end{array}
        \right) = \left(
          \begin{array}{cc}
            0 & -1 \\
           1& 0 \\
          \end{array}
        \right).$$
 Moreover, we have $\mu([a,x])= \chi_{[0,2)}$ and $\norm{[a,x]}_p = 2^{1/p}$.
 However, $\mu(a\oplus a) =\frac12\chi_{[0,4)}$ and $\norm{a\oplus a}_p =  2^{\frac{2}{p} -1}$.
 We have
 $$ \norm{a\oplus a}_p< \norm{[a,x]}_{p} \le \norm{\delta_a}_{(\mathbb{M}_2,\norm{\cdot}_{\mathbb{M}_2})\to (L_p(\mathbb{M}_2,\tau), \norm{\cdot}_p)}.  $$

Therefore, Theorem \ref{thm2}  does not hold for general non-positive operators. 
\end{remark}

 {\bf Acknowledgement:} The authors would like to thank  Aleksey Ber for helpful discussions and comments. 
  
\bibliographystyle{amsalpha}

\end{document}